\documentclass[fleqn,reqno,11pt,a4paper,final]{amsart}

\usepackage[a4paper,left=20mm,right=20mm,top=20mm,bottom=20mm,marginpar=20mm]{geometry}
\usepackage{amsmath}
\usepackage{amssymb}
\usepackage{amsthm}
\usepackage{mathtools}
\usepackage{amscd}
\usepackage[ansinew]{inputenc}
\usepackage[noadjust]{cite}
\usepackage{bbm}
\usepackage{color}
\usepackage[english=american]{csquotes}
\usepackage[final]{graphicx}
\usepackage{hyperref}
\usepackage{calc}
\usepackage{mathptmx}
\usepackage{tikz}
\usepackage{graphicx}
\usepackage{stix}

\graphicspath{{../Pictures/}}

\numberwithin{equation}{section}

\newtheoremstyle{thmlemcorr}{10pt}{10pt}{\itshape}{}{\bfseries}{.}{10pt}{{\thmname{#1}\thmnumber{ #2}\thmnote{ (#3)}}}
\newtheoremstyle{thmlemcorr*}{10pt}{10pt}{\itshape}{}{\bfseries}{.}\newline{{\thmname{#1}\thmnumber{ #2}\thmnote{ (#3)}}}
\newtheoremstyle{remexample}{10pt}{10pt}{}{}{\bfseries}{.}{10pt}{{\thmname{#1}\thmnumber{ #2}\thmnote{ (#3)}}}
\newtheoremstyle{ass}{10pt}{10pt}{}{}{\bfseries}{.}{10pt}{{\thmname{#1}\thmnumber{ A#2}\thmnote{ (#3)}}}

\theoremstyle{thmlemcorr}
\newtheorem{theorem}{Theorem}
\numberwithin{theorem}{section}
\newtheorem{lemma}[theorem]{Lemma}
\newtheorem{corollary}[theorem]{Corollary}
\newtheorem{proposition}[theorem]{Proposition}

\theoremstyle{thmlemcorr*}
\newtheorem{theorem*}{Theorem}
\newtheorem{lemma*}[theorem]{Lemma}
\newtheorem{corollary*}[theorem]{Corollary}
\newtheorem{proposition*}[theorem]{Proposition}
\newtheorem{problem*}[theorem]{Problem}
\newtheorem{conjecture*}[theorem]{Conjecture}
\newtheorem{definition*}[theorem]{Definition}
\newtheorem{assumption*}[theorem]{Assumption}

\theoremstyle{remexample}
\newtheorem{remark}[theorem]{Remark}

\theoremstyle{ass}

\newcommand{\Acal}{\mathcal{A}}

\newcommand{\Fcal}{\mathcal{F}}

\newcommand{\Hcal}{\mathcal{H}}

\newcommand{\Lcal}{\mathcal{L}}

\newcommand{\N}{\mathbb{N}}
\newcommand{\R}{\mathbb{R}}

\newcommand{\eps}{\epsilon}

\def\XXint#1#2#3{{\setbox0=\hbox{$#1{#2#3}{\int}$}
\vcenter{\hbox{$#2#3$}}\kern-.5\wd0}}

\renewcommand{\eps}{\varepsilon}
\renewcommand{\epsilon}{\varepsilon}
\renewcommand{\phi}{\varphi}

\begin{document}


\title[]{Local strong solutions to a quasilinear degenerate fourth-order thin-film equation}

\author{Christina Lienstromberg}
\address{\textit{Christina Lienstromberg:} Institute of Applied Mathematics, University of Bonn, Endenicher Allee~60, 53115 Bonn, Germany}
\email{lienstromberg@iam.uni-bonn.de}

\author{Stefan M\"uller}
\address{\textit{Stefan M\"uller:} Institute of Applied Mathematics, University of Bonn, Endenicher Allee~60, 53115 Bonn, Germany}
\email{stefan.mueller@hcm.uni-bonn.de}

\begin{abstract}
We study the problem of existence and uniqueness of strong solutions to a 
degenerate quasilinear parabolic non-Newtonian thin-film equation. 
Originating from a non-Newtonian Navier--Stokes system the equation is derived
by lubrication theory and under the assumption that capillarity is the only driving force.
The fluid's shear-thinning rheology is described by the so-called Ellis constitutive law.
For flow behaviour exponents $\alpha \geq 2$ the corresponding 
initial boundary value problem fits into the abstract setting of 
\cite[Thm. 12.1]{A:1993}. Due to a lack of regularity this is not true
for flow behaviour exponents $\alpha \in (1,2)$. 
For this reason we prove an existence theorem for abstract quasilinear 
parabolic evolution problems with H\"older continuous dependence.
This result provides existence of strong solutions to the non-Newtonian thin-film 
problem in the setting of fractional Sobolev spaces and (little) H\"older spaces.
Uniqueness of strong solutions is derived by energy methods and by using the
particular structure of the equation.
\end{abstract}
\vspace{4pt}







\maketitle

\section{Introduction}
This contribution is motivated by questions for existence and uniqueness of
strong solutions to the degenerate quasilinear fourth-order evolution problem
\begin{equation} \label{eq:thin-film_intro}
	\begin{cases}
		u_t +  a \bigl(u^3 \bigl[1 + |b u u_{xxx}|^{\alpha-1}\bigr] u_{xxx}\bigr)_x
		=
		0,
		&
		\quad
		t > 0, \, x \in \Omega := (-l,l)
		\\
		u_x = u_{xxx}
		=
		0,
		&
		\quad
		t > 0, \, x \in \partial\Omega
		\\
		u(0,\cdot)
		=
		u_0(\cdot),
		&
		\quad
		x \in \Omega,
	\end{cases}
\end{equation}
describing the evolution of the height $u(t,x)$ of a non-Newtonian incompressible thin liquid film
on a solid bottom. Here $a, b >0$ denote positive constants that depend on the fluid's properties 
and are specified later and $\alpha$ describes the fluid's flow behaviour.
Problem \eqref{eq:thin-film_intro} is derived by applying lubrication theory \cite{O:1998,GO:2003}
to the non-Newtonian Navier--Stokes system. We use the so-called Ellis constitutive law 
\cite{WS:1994,MB:1965} to describe the fluid's shear-thinning rheology. Observe that the fluid is
Newtonian if $\alpha = 1$. In this case we recover from $\eqref{eq:thin-film_intro}_1$ the well-known
thin-film equation
\begin{equation*}
	u_t + a (u^3 u_{xxx})_x = 0,
	\quad
	t > 0, \, x \in \Omega.
\end{equation*}
In this paper we treat the case $\alpha > 1$ in which the fluid is said to be shear-thinning. 
Moreover, we assume the fluid's dynamics to be driven by capillary forces only. 
In particular gravitational forces are neglected; c.f. Section \ref{sec:model} for a more detailed 
review of the derivation of the governing equations. 

\medskip

There is a rich literature on weak solutions of equations related to \eqref{eq:thin-film_intro}.
The first pioneering results on the classical Newtonian thin-film equation  
\begin{equation} \label{eq:TF} 
	u_t + a (u^n u_{xxx})_x = 0,
	\quad
	n \in \N,
\end{equation}
go back to \textsc{Bernis} and \textsc{Friedmann} \cite{BF:1990}. There the authors prove
in particular the existence of global non-negative weak solutions as well as positivity and
uniqueness for $n \geq 4$. Among others, also the works \cite{BBdP:1995,BP:1996}
can be mentioned in the context of global non-negative weak solutions.
For contributions to the study of non-Newtonian fluids the reader shall be in particular referred to
the works \cite{K:2001a,King:2001b} of \textsc{King} where non-Newtonian generalisations of \eqref{eq:TF} 
are investigated. In particular the author studies the doubly nonlinear equation
\begin{equation} \label{eq:doubly_nonl}
	u_t + \bigl(u^n |u_{xxx}|^{p-2} u_{xxx}\bigr)_x = 0,
\end{equation}
describing for $p\neq2$ the spreading of so-called power-law or Ostwald--de Waele fluids. 
In the work \cite{AG:2004} \textsc{Ansini} and \textsc{Giacomelli} establish the existence of global
non-negative weak solutions to \eqref{eq:doubly_nonl}
for $p > 2$ and $\frac{p-1}{2} < n < 2p-1$. 
In \cite{AG:2002} the same authors verify the existence of travelling-wave solutions and study a class 
of quasi-self-similar solutions to the equation \eqref{eq:thin-film_intro}. 
Moreover, in \cite{MRZ:2016} the authors establish 
the existence of weak solutions to a non-Newtonian Stokes equation with a viscosity that depends on the 
fluid's shear rate and its pressure at the same time.

\medskip

In this paper we focus on strong solutions. To construct such solutions 
it is convenient to give up the divergence form of \eqref{eq:thin-film_intro} in order
to emphasise the quasilinear structure. For sufficiently regular solutions
\eqref{eq:thin-film_intro} is equivalent to
\begin{equation} \label{eq:thin-film_non-div_intro}
	u_t + A(u, u_x, u_{xx}, u_{xxx})\, u_{xxxx} = F(u, u_{x}, u_{xx}, u_{xxx}),
	\quad
	t > 0,\ x \in \Omega,
\end{equation}
where
\begin{equation*}
	A(u, u_x, u_{xx}, u_{xxx}) = a u^3 \bigl(1 + \alpha |buu_{xxx}|^{\alpha-1}\bigr)
	\quad
	\text{and}
	\quad
	F(u, u_{x}, u_{xx}, u_{xxx}) = -3 a u^2 \bigl(1 + |\tilde{b} u u_{xxx}|^{\alpha-1}\bigr) u_x u_{xxx}
\end{equation*}
In \eqref{eq:thin-film_non-div_intro} the highest-order term $u_{xxxx}$ appears
only linearly, while the nonlinearity $|u_{xxx}|^{\alpha-1} u_{xxx}$ is on 
the right-hand side. However, the coefficient function $A$ contains the delicate
nonlinearity $|u_{xxx}|^{\alpha-1}$.

Regarding the existence and uniqueness of strong solutions to
\eqref{eq:thin-film_intro} it turns out that there is a qualitative difference 
between flow behaviour exponents $\alpha \in (1,2)$ and those larger
than or equal $2$. 
If we associate to \eqref{eq:thin-film_non-div_intro} an abstract quasilinear Cauchy problem 
\begin{equation*}
	\begin{cases}
		\dot{u} + \Acal(u) u = \Fcal(u), & \quad t > 0
		\\
		u(0) = 0, &
	\end{cases}
\end{equation*}
it turns out that in the latter case $\alpha \geq 2$ the operator $\Acal$ and the 
right-hand side $\Fcal$ are Lipschitz continuous in an appropriate sense and 
the classical H\"older theory of \textsc{Eidel'man} \cite[Thm. III.4.6.3]{E:1969} as well as
the abstract theory of \textsc{Amann} \cite[Thm. 12.1]{A:1993} are applicable and provide
existence and uniqueness of strong solutions. 

\medskip

The situation is more delicate for flow behaviour exponents $\alpha \in (1,2)$. In this 
regime the operator $\Acal$ and the right-hand side $\Fcal$ are only $(\alpha-1)$-H\"older
continuous, whence there is no hope to obtain existence and uniqueness by 
Banach's fixed point theorem. Instead we use compactness to prove an abstract 
existence result for quasilinear parabolic problems of fourth order with H\"older continuous 
coefficients in the spirit of \cite[Thm. 12.1]{A:1993}. More precisely, the proof exploits 
the a-priori estimates and the smoothing properties of the corresponding abstract linear equation
to obtain a solution for the quasilinear problem by a fixed-point argument. 

\medskip

Finally we apply this result to obtain existence of solutions to \eqref{eq:thin-film_intro}
in (fractional) Sobolev spaces as well as in (little) H\"older spaces.

\medskip

In order to prove uniqueness of solutions to \eqref{eq:thin-film_intro} 
we explicitly retain its divergence form. Indeed, we use the special
structure of the equation to obtain uniqueness by energy methods.This 
idea has already been used in the pioneering work \cite{BF:1990} of 
\textsc{Bernis} \& \textsc{Friedman} on the (Newtonian) thin-film equation.

\medskip

We close this introduction by briefly outlining the organisation of this work.

A derivation of the evolution problem is reviewed in Section \ref{sec:model}.
In Section \ref{sec:LP} we recall an abstract well-posedness result of \textsc{Amann} 
\cite{A:1995} for linear parabolic problems. 
Section \ref{sec:QP} is concerned with an existence result for abstract quasilinear 
parabolic equations.

In Section \ref{sec:application_existence} this result is applied to the non-Newtonian thin-film equation 
for all $\alpha > 1$ and in the regime of fractional Sobolev and little H\"older spaces, respectively.
Uniqueness of strong solutions to the non-Newtonian thin-film equation for flow behaviour 
exponents $\alpha \in (1,2)$ is proved in Section \ref{sec:uniqueness}. 
There we also show that constants are the only possible
steady state solutions of \eqref{eq:thin-film_intro}.
Finally, Section \ref{sec:max_existence_time} contains a result on the maximal
existence time of solutions to \eqref{eq:thin-film_intro}.
%

\section{Derivation of the non-Newtonian thin-film problem for an Ellis shear-thinning fluid} 
\label{sec:model}

For the convenience of the reader in this section we review the derivation of 
\eqref{eq:thin-film_intro}.

\medskip

\textsc{Lubrication approximation.} \hspace{0.2cm}
Starting from the non-Newtonian Navier--Stokes equations with no-slip boundary condition, 
the lubrication approximation \cite{GO:2003,O:1998}
leads -- under the assumption of a positive film height --  to the system 
\begin{alignat}{2}
	-p_x + \bigl(\mu(|v_z|) v_z\bigr)_z
	&=
	0 
	&& 
	\text{in } \Lambda \notag\\
	p_z 
	&=
	0
	&&
	\text{in } \Lambda \notag \\
	v_x + w_z
	&=
	0
	&& 
	\text{in } \Lambda \notag \\
	v = w
	&=
	0
	&& 
	\text{on } z = 0 \label{eq:No-slip} \\
	p
	&=
	- \sigma u_{xx}
	\quad
	&&
	\text{on } z = u(t,x) \notag \\
	v_z
	&=
	0
	\quad
	&&
	\text{on } z = u(t,x) \notag \\
	u_t + v u_x
	&=
	w
	&& 
	\text{on } z = u(t,x) \notag
\end{alignat}
of dimensionless equations for the velocity field $(v,w)=(v(t,x,z),w(t,x,z))$, 
the pressure $p=p(t,x,z)$ and the film height $u=u(t,x)$. Here, $\sigma > 0$ 
denotes the constant surface tension and $\mu: \R \to \R$ the fluid's viscosity. The fluid 
domain $\Omega$ is defined by
\begin{equation*}
	\Lambda := \left\{(x,z) \in \R^2; 0 < x < l \text{ and } 0 < z < u(t,x)\right\}.
\end{equation*}
Integrating $\eqref{eq:No-slip}_2$ from $z$ to $u$ and using the boundary condition 
$\eqref{eq:No-slip}_5$ we obtain
\begin{equation*}
	p(x,z) 
	=
	p(x,u) 
	=
	-\sigma u_{xx},
	\quad
	x \in \Omega.
\end{equation*}
Moreover, an integration of $\eqref{eq:No-slip}_1$ from $z$ to $u$, together with the boundary condition
$\eqref{eq:No-slip}_6$ yields
\begin{equation} \label{eq:mu_uz}
	\mu(|v_z|) v_z
	=
	\sigma u_{xxx} (u-z),
	\quad
	(x,z) \in \Lambda.
\end{equation}

\medskip


\textsc{Shear-thinning rheology: Ellis constitutive law.} \hspace{0.2cm}
As for instance in \cite{AG:2002,WS:1994} we use the Ellis constitutive law
to describe the shear-thinning rheology. 
We thus introduce the shear stress $\tau: \R \to \R$ implicitly by the relation
\begin{equation}\label{eq:Ellis1}
	s = \frac{1}{\mu_0} \left(1 + \left|\frac{\tau(s)}{\tau_{\ast}}\right|^{\alpha - 1}\right) \tau(s)
\end{equation}
and set 
\begin{equation} \label{eq:Ellis2}
	\mu(s) = \frac{\tau(s)}{s}.
\end{equation}
Here $\mu_0$ denotes the viscosity at zero shear stress and $\tau_{\ast}$ is the shear
stress at which the viscosity is reduced by a factor $1/2$. Shear-thinning fluids are characterised 
by an apparent viscosity that decreases with increasing shear stress. For $\alpha > 1$ and 
$\tau_{\ast} \in (0,\infty)$ the shear-thinning behaviour is reflected by the Ellis law. Moreover, it
is observed in polymeric systems that at rather low and/or rather high shear rates the
viscosity approaches a Newtonian plateau. For most polymers and polymer solutions 
$\alpha$ varies from $1$ to $2$ (see i.e. \cite{BAH:1977,MB:1965}). This Newtonian
plateau may not be reflected by Ostwald--de Waele (pure power-law) fluids which gives rise
to the use of other models as for instance the Ellis law. Observe that in \eqref{eq:Ellis1}--\eqref{eq:Ellis2}
the Newtonian plateau may be recovered for $\alpha = 1$ or $1/\tau_{\ast} \to 0$. 

\medskip

Using \eqref{eq:Ellis1} and \eqref{eq:Ellis2} in \eqref{eq:mu_uz} leads to the equation
\begin{equation*}
	v_z(x,z)
	=
	\frac{\sigma}{\mu_0} u_{xxx} (u-z) 
	+
	\frac{\sigma^{\alpha}}{\mu_0 \tau_{\ast}^{\alpha-1}} |u_{xxx}|^{\alpha-1} u_{xxx} (u-z)^{\alpha},
	\quad
	(x,z) \in \Lambda,
\end{equation*}
whence the horizontal velocity is given by
\begin{equation*}
	v(x,z)
	=
	\frac{\sigma}{\mu_0} u_{xxx} \left(uz-\frac{z^2}{2}\right) 
	- \frac{\sigma^{\alpha}}{(\alpha+1) \mu_0 \tau_{\ast}^{\alpha-1}} 
	|u_{xxx}|^{\alpha-1} u_{xxx} \left((u-z)^{\alpha+1} - u^{\alpha+1}\right),
	\quad
	(x,z) \in \Lambda.
\end{equation*}
For $x \in (0,l)$ this finally yields
\begin{equation*}
	\int_0^u v(x,z) dz
	=
	\frac{\sigma}{3 \mu_0} u_{xxx} u^3
	+
	\frac{\sigma^{\alpha}}{(\alpha+2) \mu_0 \tau_{\ast}^{\alpha-1}} 
	|u_{xxx}|^{\alpha-1} u_{xxx} u^{\alpha+2}
\end{equation*}
and the evolution equation for the film's height $u$, cf. $\eqref{eq:No-slip}_7$, thus reads
\begin{equation*}
	u_t + a\, \bigl(u^3 \bigl[1 + |b\,u\, u_{xxx}|^{\alpha-1}\bigr] u_{xxx}\bigr)_x
	=
	0,
	\quad
	t > 0, x \in \Omega,
\end{equation*}
with
\begin{equation} \label{eq:a_b}
	a :=
	\frac{\sigma}{3 \mu_0}
	\quad
	\text{and}
	\quad
	b :=
	\left(\frac{3}{\alpha+2}\right)^{\frac{1}{\alpha-1}} \frac{\sigma}{\tau_{\ast}}.
\end{equation}

\medskip


\section{A well-posedness result for abstract linear parabolic initial value problems} \label{sec:LP}

This section is concerned with the solvability of linear parabolic 
evolution problems.
We mainly recall a basic well-posedness and regularity result 
for linear parabolic Cauchy problems. 
We start by introducing some notations and requirements.

\medskip

In the following let $E_0$ and $E_1$ be Banach spaces. We write
\begin{equation*}
	E_1 \xhookrightarrow{d} E_0
	\quad
	\text{and}
	\quad
	E_1 \xhookrightarrow{c} E_0
\end{equation*}
if the continuous embedding of $E_1$ in $E_0$ is in addition dense, 
respectively compact. We say that $(E_0, E_1)$ is a densely injected Banach couple if $E_1 \xhookrightarrow{d} E_0$.
Given a densely injected Banach couple $(E_0,E_1)$
and $\theta \in (0,1)$, we denote by $(\cdot,\cdot)_{\theta}$ and 
$[\cdot,\cdot]_{\theta}$ a real, respectively the complex interpolation functor
and set $E_{\theta} = (E_0,E_1)_{\theta}$, respectively $E_{\theta} = [E_0,E_1]_{\theta}$
with norm $||\cdot||_{E_{\theta}}$.
It is well-known \cite[Thm. I.2.11.1]{A:1995} that for a densely injected Banach couple $(E_1,E_0)$ we have
\begin{equation*}
	E_{\alpha} \xhookrightarrow{d} E_{\beta},
	\quad
	0 \leq \beta < \alpha \leq 1.
\end{equation*}
Similarly, if $E_1 \xhookrightarrow{c} E_0$, then
\begin{equation*}
	E_{\alpha} \xhookrightarrow{c} E_{\beta},
	\quad
	0 \leq \beta < \alpha \leq 1.
\end{equation*}

\medskip

Now let $(E_0, E_1)$ be a densely injected Banach couple and let  $T > 0$ be given. 
For each $t \in [0,T]$ let $\Acal(t)$ be a linear operator in $E_0$ with 
domain $E_1$ and let $\Fcal : [0,T] \to E_0$. We consider the linear Cauchy problem
\begin{equation}\label{eq:LP}
	\begin{cases}
		\dot{u} + \Acal(t) u &=\quad \Fcal(t) \quad \text{in } (0,T)
		\\
		u(0) &=\quad u_0.
	\end{cases}
\end{equation}
We call \eqref{eq:LP} parabolic, if $-\Acal(t)$ is for each $t \in [0,T]$ the infinitesimal generator
of a strongly continuous analytic semigroup on $E_0$, in symbols $\Acal \in \Hcal(E_1;E_0)$.
We equip $\Hcal(E_1;E_0)$ with the operator norm $||\cdot||_{\Lcal(E_1;E_0)}$.

\medskip

To derive estimates we need a quantitative description of $\mathcal H(E_1, E_0)$. 
Given $\kappa \ge 1$ and $\omega > 0$ we write
\begin{equation*}
	A \in \mathcal H(E_1, E_0, \kappa, \omega)
\end{equation*}
if $\omega + A$ is an isomorphism from $E_1$ to $E_0$  and 
\begin{equation}
	\kappa^{-1} \le \frac{\| \lambda + A \|_{E_0}}{|\lambda \|x\|_{E_0} + \| x\|_{E_1}} 
	\le \kappa \quad \text{whenever} \quad  x \in E_1\setminus\{0\}, \quad \Re \lambda \ge \omega.
\end{equation}
We have \cite[Thm. I.1.2.2]{A:1995}
\begin{equation} \label{eq:calH}
	\mathcal H(E_1, E_0) = \bigcup_{\substack{ \kappa \ge 1\\ \omega > 0}}  \mathcal H(E_1, E_0, \kappa, \omega).
\end{equation}

\medskip

We say that $u$ is a solution of \eqref{eq:LP} (in $E_0$) if $u \in C^1( (0,T]; E_0)$ such that $u(t) \in E_1$ for $t \in (0, T]$,
$u \in C([0,T]; E_0)$ with $u(0) = u_0 \in E_0$ and the differential equation holds in $(0, T)$. 
We recall the following fundamental well-posedness and regularity result 
for linear parabolic Cauchy problems \cite[Thm. II.1.2.1 and Thm. II.5.3.1]{A:1995}.

\medskip

\begin{theorem} \label{thm:existence_linear_Amann}
Let $(E_0,E_1)$ be a densely injected Banach couple and suppose that
\begin{equation} \label{eq:assumption_A_F}
	\Fcal \in C^{\rho}([0,T];E_0)
	\quad
	\text{and}
	\quad
	\Acal \in C^{\rho}([0,T];\mathcal{H}(E_1;E_0))
\end{equation}
for some $\rho \in (0,1)$. Then the following holds true.
\begin{itemize}
	\item[(i)]     \label{it:linear1}   If $u_0 \in E_0$ then the linear Cauchy problem \eqref{eq:LP} possesses a unique solution
		\begin{equation*}
			 u \in C^{\rho}((0,T);E_1) \cap C^{1+\rho}((0,T);E_0) \cap C([0,T]; E_0)
		\end{equation*}
	\item[(ii)]     \label{it:linear2} If $u_0 \in E_1$, then \eqref{eq:LP} possesses a unique solution
		\begin{equation*}
			u \in C([0,T];E_1) \cap C^1([0,T];E_0).
		\end{equation*}
	\item[(iii)] If $\alpha \in (0,1)$ and $u_0 \in E_\alpha$ then the unique solution in (i)  satisfies in addition
		\begin{equation} \label{eq:linear_C_E_alpha}
			u \in C([0, T]; E_\alpha).
		\end{equation}
\end{itemize}
Let now $0 \leq \beta \leq \alpha < 1$ and $\Fcal \in L_{\infty,{\text{loc}}}([0,T];E_0)$,
and assume that there exists a constant $C > 0$ such that
\begin{equation} \label{eq:A_uniformly_Holder}
	\Acal \in C^{\rho}([0,T];\Lcal(E_1;E_0))
	\quad
	\text{with}
	\quad
	[\Acal(\cdot)]_{C^{\rho}([0,T];\Lcal(E_1;E_0))} \leq C.
\end{equation}
Moreover assume that there are constants $\omega > 0, \kappa \geq 1$ and $\sigma \in \R$ such that
\begin{equation*}
	\sigma + \Acal \in C\bigl([0,T];\Hcal(E_1,E_0,\kappa,\omega)\bigr).
\end{equation*}
Then we have in addition that
\begin{itemize}
	\item[(iv)] if $u_0 \in E_{\alpha}$ then $u \in C^{\alpha-\beta}([0,T];E_{\beta})$ with
		\begin{equation}  \label{eq:hoelder_estimate_linear}
			||u(t) - u(s)||_{E_{\beta}}
			\leq 
			C (t-s)^{\alpha-\beta} e^{\nu T} \bigl(||u_0||_{E_{\alpha}} + ||\Fcal||_{L_{\infty}([0,T];E_0)}\bigr),
			\quad
			t, s \in [0,T].
		\end{equation}
		The constants $C, \nu > 0$ do not depend on $T$.
\end{itemize}
\end{theorem}

\medskip

A proof of this theorem can be found in the book \cite{A:1995}. 
While part (i) and part (iv) are due to \textsc{Amann},
part (ii) goes back to \textsc{Sobolevskii} \cite{S:1966} and \textsc{Tanabe} \cite{T:1960}.
Moreover, there is an analogous result by \textsc{Acquistapace \& Terreni} \cite{AT:1985}
which is proved by different methods.
As regards assertion (iii),  note that $\mathcal A([0,T])$ is compact subset of $\mathcal H(E_1, E_0)$. 
Thus by  \cite[Cor. I.1.3.2]{A:1995}. there exist $\kappa \ge 1$ and $\omega >0$ such that 
$ A([0,T]) \subset H(E_1, E_0, \kappa, \omega)$. Now  \eqref{eq:linear_C_E_alpha}
follows from \cite[Thm. II.5.3.1]{A:1995} applied with $\beta = \alpha$. 

\medskip


\section{Abstract existence theorem for the quasilinear problem} 
\label{sec:QP}

In this section we prove an existence result for abstract
quasilinear parabolic Cauchy problems based on the theory for 
linear parabolic problems. This result will later be applied to the 
non-Newtonian thin film equation in the setting of fractional Sobolev spaces
and (little) H\"older spaces.

\medskip

As in Section \ref{sec:LP} let $T > 0$ be given and consider the quasilinear Cauchy problem
\begin{equation}\label{eq:QP}
	\begin{cases}
		\dot{u} + \Acal(u) u &=\quad \Fcal(u) \quad \text{in } (0,T)
		\\
		u(0) &=\quad u_0,
	\end{cases}
\end{equation}
where $\Acal$ and $\Fcal$ are H\"older continuous in an appropriate sense, to be specified below.

\medskip

The existence result for \eqref{eq:QP} is deduced from the uniform 
a-priori estimates for the corresponding linear equation
and an application of the following fixed-point theorem \cite[Cor. 11.2]{GT:1998}. 

\medskip

\begin{theorem} \label{thm:fp}
Let $X$ be a Banach space. Let $B \subset X$ be closed, convex and not empty. 
If $S : B \to B$ is continuous and compact, then $S$ has a fixed point.
\end{theorem}

\medskip

It is worthwhile to mention again that with Theorem \ref{thm:fp} we may obtain 
an existence result for \eqref{eq:QP} by compactness of the solution operator
for the corresponding linear problem. Since we do not require Lipschitz continuity
we can in general not expect to get uniqueness.  

\medskip

\begin{theorem}\label{th:existence_quasilinear_abstract} 
Suppose that $(E_0, E_1)$ is a densely injected Banach couple and that
the injection is in addition compact.
Let $0 < \beta < \alpha \le 1$ and $\sigma \in (0,\alpha-\beta)$. 
Let $\mu \in (0,1)$ and assume that the maps
\begin{eqnarray}
	\mathcal A: E_\beta &\to& \Hcal(E_1;E_0), \\
	\mathcal F: E_\beta &\to& E_0
\end{eqnarray}
are $\mu$-H\"older continuous on all balls $B_{\beta}(0,R)$ in $E_{\beta}$.
Then for each $R' > 0$ there exist a positive time $T > 0$ with the following property. 
If $u_0 \in E_\alpha$ and $\| u_0\|_{E_\alpha} \le R'$ then 
the quasilinear initial value problem \eqref{eq:QP}
possesses a solution  
\begin{equation*}
	u \in C^{\mu\sigma}((0,T];E_1) \cap C^{1+\mu\sigma}((0,T];E_0) \cap C([0, T]; E_\alpha).
\end{equation*}
Moreover, for all $\alpha' \in (\beta+\sigma,\alpha)$ the solution satisfies
\begin{equation*}
	u \in C^{\sigma}([0,T];E_{\alpha'-\sigma}) \cap C^{\alpha'-\beta}([0,T];E_{\beta}).
\end{equation*}
\end{theorem}

\medskip

\begin{proof}
Fix $\alpha$, $\beta$ and $R'$. We will use a fixed-point argument in a suitable ball in the space 
$C^\rho([0, T]; E_\beta)$ for a suitable $\rho >0$. The argument uses
the estimate   \eqref{eq:hoelder_estimate_linear}. Thus we start with the following observation. 
There exist $\omega > 0$, $\kappa \ge 1$ and $r_0 > 0$
such that for all $w \in E_\beta$ we have the following implication:
\begin{equation} \label{eq:uniform_kappa_omega}
	\text{$\|u_0 \|_{E_\alpha} \le R'$ 
	and $\| w - u_0 \|_{E_\beta} \le r_0$}   
	\quad \Longrightarrow \quad 
	\mathcal A(w) \in \mathcal H(E_1, E_0, \kappa, \omega).
\end{equation}
To see this, note first that the set $K = \{ u_0 \in E_\alpha : \| u_0 \|_{E_\alpha} \le R'\}$ is compact in $E_\beta$. 
Hence $\mathcal A(K)$ is a compact set in $\mathcal H(E_1, E_0)$. By \eqref{eq:calH} and a simple perturbation 
argument, see \cite[Thm. I.1.3.1]{A:1995},  this implies that there exist an open set $V$ and $ \omega > 0$, $\kappa \ge 1$
such that $\Acal(K) \subset V \subset \mathcal H(E_1, E_0, \kappa, \omega)$. 
Let $U$ be the preimage $U = \mathcal A^{-1}(V)$. Then $U$ is open in $E_\beta$  and $U \supset K$. 
Since $K$ is compact it follows that $U$ contains a $2 r_0$ neighbourhood of $K$ for some $r_0 > 0$. Thus  
\eqref{eq:uniform_kappa_omega} holds.

Now let $u_0 \in E_{\alpha}$ such that $||u_0||_{E_{\alpha}} \leq R^{\prime}$. 
By $\bar{u}_0$ we denote the constant extension of $u_0$ on $[0,T]$.
We set up a fixed-point problem suitable for an application of Theorem \ref{thm:fp} as follows.
We set
\begin{equation*}
	X := C^{\sigma}([0,T];E_{\beta})
	\quad
	\text{and}
	\quad
	\bar B := \bar{B}(\bar{u}_0,r_0) := \{v \in X; ||v - \bar{u}_0||_X \leq r_0\}.
\end{equation*}
Given $v \in \bar{B}$, the regularity assumptions on $\Acal$ and $\Fcal$ imply 
\begin{equation*}
	\Acal(v(\cdot)) \in C^{\mu\sigma}\bigl([0,T];\Hcal(E_1;E_0)\bigr)
	\quad
	\text{and}
	\quad
	\Fcal(v(\cdot)) \in C^{\mu\sigma}([0,T];E_0).
\end{equation*}
Furthermore, for $v \in \bar{B}$ we have 
\begin{align*}
	[\Acal(v(\cdot))]_{C^{\mu\sigma}([0,T];\Lcal(E_1;E_0))}
	& \leq
	[\Acal]_{C^{\mu}(\bar{B}_{\beta}(0,R'+r_0);\Lcal(E_1;E_0))} [v]_{C^{\sigma}([0,T];E_{\beta})}^{\mu}
	\\
	&\leq
	[\Acal]_{C^{\mu}(\bar{B}_{\beta}(0,R'+r_0);\Lcal(E_1;E_0))} [v - \bar{u}_0]_{C^{\sigma}([0,T];E_{\beta})}^{\mu}
	\\
	&\leq
	[\Acal]_{C^{\mu}(\bar{B}_{\beta}(0,R'+r_0);\Lcal(E_1;E_0))} r_0^{\mu},
\end{align*}
i.e. a uniform bound of the $\mu\sigma$-H\"older norm of $\Acal$ on $[0,T]$. 
Moreover, recalling that $u_0 \in \bar{B}_{\alpha}(0,R') \subset E_{\alpha}$ by assumption,
we deduce from \eqref{eq:uniform_kappa_omega} the existence of some 
$\kappa \geq 1$ and $\omega > 0$ such that
\begin{equation*}
	\Acal(v(t)) \in \Hcal(E_1,E_0,\kappa,\omega),
	\quad
	t \in [0,T],
\end{equation*}
for all $v \in \bar{B}$.
Finally, observe that for all $v \in \bar{B}$ we have the estimate
\begin{equation} \label{eq:F_uniform}
\begin{split}
	||\Fcal(v)||_{L_{\infty}([0,T];E_0)}
	& \leq
	||\Fcal(u_0)||_{E_0} + [\Fcal]_{C^{\mu}(\bar{B}_{\beta}(0,R'+r_0);E_0)} 
	||v - \bar{u}_0||_{C^{\sigma}([0,T];E_{\beta})}^{\mu}
	\\
	& \leq
	||\Fcal(u_0)||_{E_0} + [\Fcal]_{C^{\mu}(\bar{B}_{\beta}(0,R'+r_0);E_0)} r_0^{\mu}.
\end{split}
\end{equation}

Thus, we are in a position to apply Theorem \ref{thm:existence_linear_Amann} to deduce existence of
a unique solution $u = Sv$ of the linear evolution problem
\begin{equation}\label{eq:LP_proof}
	\begin{cases}
		\dot{u} + \Acal(v(t))\, u &=\quad \Fcal(v(t)) \quad \text{in } (0,T]
		\\
		u(0) &=\quad u_0
	\end{cases}
\end{equation}
in the sense that $u = Sv$ satisfies
\begin{equation*}
	u \in C^{\mu\sigma}((0,T];E_1) \cap C^{1+\mu\sigma}((0,T];E_0) \cap C([0,T];E_{\alpha}).
\end{equation*}
Let now $\alpha^{\prime} \in (\beta+\sigma,\alpha)$ and note that 
$u_0 \in E_{\alpha} \hookrightarrow E_{\alpha^{\prime}}$. Then the solution
enjoys in addition the regularity 
\begin{equation*}
	u \in C^{\sigma}([0,T];E_{\alpha^{\prime}-\sigma}) \cap C^{\alpha^{\prime}-\beta}([0,T];E_{\beta}) 
\end{equation*}
and by \eqref{eq:hoelder_estimate_linear} we have the estimates
\begin{equation*}
	||u||_{C([0,T];E_{\alpha})} 
	+
	[u]_{C^{\sigma}([0,T];E_{\alpha^{\prime}-\sigma})} 
	+
	[u]_{C^{\alpha^{\prime}-\beta}([0,T];E_{\beta})}
	\leq
	C e^{\nu T} \bigl(||u_0||_{E_{\alpha}} + ||\Fcal(v)||_{L_{\infty}([0,T];E_0)}\bigr),
\end{equation*}
where $\nu \geq 0$ is independent of $T$. 
In view of the uniform estimate \eqref{eq:F_uniform} on $\Fcal$ the solution $u=Sv$ satisfies in fact the 
uniform estimate
\begin{equation} \label{eq:linear_estimates}
	||u||_{C([0,T];E_{\alpha})} 
	+
	[u]_{C^{\sigma}([0,T];E_{\alpha^{\prime}-\sigma})} 
	+
	[u]_{C^{\alpha^{\prime}-\beta}([0,T];E_{\beta})}
	\leq
	C(R',r_0)\, e^{\nu T},
\end{equation}
where the constants $C$ and $\nu$ are independent of $T$.

In order to deduce from Theorem \ref{thm:fp} the existence of a fixed point $u = Su \in X$
we are thus left with verifying that
\begin{itemize}
	\item[(i)] $S: \bar{B} \to X$ is continuous;
	\item[(ii)] $S: \bar{B} \to X$ is compact;
	\item[(iii)] $S$ preserves the ball $\bar{B}$.
\end{itemize}

(i) Continuity of $S$. From the linear theory we know that $S$ maps the space $\bar{B}$ continuously into
\begin{equation*}
	Y := C^{\sigma}([0,T];E_{\alpha^{\prime}-\sigma}) \cap C^{\alpha^{\prime}-\beta}([0,T];E_{\beta})
\end{equation*} 
Since $\alpha^{\prime} - \sigma >  \beta$ we have $E_{\alpha^{\prime}-\sigma} \hookrightarrow E_{\beta}$
and hence continuity of the embedding
\begin{equation*}
	C^{\sigma}([0,T];E_{\alpha^{\prime}-\sigma}) \hookrightarrow C^{\sigma}([0,T];E_{\beta}).
\end{equation*}

\medskip

(ii) Compactness of $S$. We know that $S$ maps $\bar{B}$ to $Y$ from the linear theory.
We show that $Y$ is compactly embedded into $X$. Since $\alpha' - \sigma > \beta$,
it follows from compactness of the 
embedding
$E_{\alpha^{\prime}-\sigma} \xhookrightarrow{c} E_{\beta}$ 
and the Arzelà--Ascoli theorem that 
\begin{equation*}
	C^{\sigma}([0,T];E_{\alpha^{\prime}-\sigma}) \xhookrightarrow{c} C([0,T];E_{\beta}).
\end{equation*}
Hence $S$ is a compact operator from $\bar{B}$ to $C([0,T];E_{\beta})$.
In view of the interpolation estimate
\begin{equation} \label{eq:interpolation_est}
	[u]_{C^{\sigma}([0,T];E_{\beta})}
	\leq
	||u||_{C([0,T];E_{\beta})}^{1-\theta}
	[u]_{C^{\alpha^{\prime}-\beta}([0,T];E_{\beta})}^{\theta}
	\quad
	\text{with }
	\theta = \frac{\sigma}{\alpha^{\prime} - \beta}
\end{equation}
we see that $S$ is even a compact operator from $\bar{B}$ to $X$.

(iii) $S(\bar{B}) \subset \bar{B}$. To deduce that $S$ has a fixed point we are 
left with verifying that for sufficiently small $T > 0$ the operator $S$ maps the ball 
$\bar{B} \subset X$ into itself. 
Recall from \eqref{eq:linear_estimates} that given $v \in \bar{B}$, we obtain the estimate 
\begin{equation*}
	[u]_{C^{\alpha^{\prime}-\beta}([0,T];E_{\beta})}
	\leq
	C(R',r_0)\, e^{\nu T}
\end{equation*}
for the solution $u=Sv$ of the linear problem \eqref{eq:LP_proof}. This implies on
the one hand that
\begin{equation} \label{eq:ball_1}
	||u - \bar{u}_0||_{C([0,T];E_{\beta})} 
	\leq
	C(R',r_0) T^{\alpha' - \beta} e^{\nu T}.
\end{equation}
and on the other hand (recall that $\sigma < \alpha' - \beta$)
\begin{equation} \label{eq:ball_2}
	[u - \bar{u}_0]_{C^{\sigma}([0,T];E_{\beta})}
	\leq
	T^{\alpha' - \beta - \sigma} [u - \bar{u}_0]_{C^{\alpha' - \beta}([0,T];E_{\beta})}
	\leq
	C(R',r_0) T^{\alpha' - \beta - \sigma} e^{\nu T}.
\end{equation}
Combining \eqref{eq:ball_1} and \eqref{eq:ball_2} we find that
\begin{equation*}
	||u - \bar{u}_0||_{C^{\sigma}([0,T];E_{\beta})}
	\leq
	C(R',r_0) (T^{\alpha' - \beta - \sigma} + T^{\alpha' - \beta})\, e^{\nu T}.
\end{equation*}
Thus for sufficiently small $T > 0$ the right-hand side of this inequality is less than or equal $r_0$ 
which proves that $S(\bar{B}) \subset \bar{B}$.
\end{proof}

\medskip

\begin{remark}
Recall that our motivation to prove this abstract result is to prove existence of
solutions to the non-Newtonian thin-film equation \eqref{eq:thin-film_intro} in different
function spaces. As mentioned in the introduction for flow behaviour exponents
$\alpha \geq 2$ for instance the results of Amman \cite[Thm. 12.1]{A:1993} and
Eidel'man \cite[III.4.6.3]{E:1969} are applicable and provide existence and
uniqueness at the same time. On the other hand there seem to be no abstract
results available for $\alpha \in (1,2)$. Roughly speaking the reason for this qualitative difference 
is that in the case $\alpha \in (1,2)$ the Nemitskii operator associated to the function 
$f(x) = |x|^{\alpha-1}$ is only $(\alpha-1)$-H\"older continuous, while it is Lipschitz
continuous for $\alpha \geq 2$. The Lipschitz continuity allows one to get
existence and uniqueness by a contraction argument.
However, concerning existence of solutions to the non-Newtonian thin-film equation 
\eqref{eq:thin-film_intro} we cover the case of flow behaviour exponents $\alpha \in (1,2)$ 
by applying our abstract existence result Theorem \ref{th:existence_quasilinear_abstract},
while we deduce uniqueness from energy estimates that use the structure of the
particular equation.
\end{remark}

\medskip

In the remainder of this section we prove a result on the maximal existence time of solutions. We use the usual continuation argument to obtain
a contradiction, but some care is required in the formulation of the result since solutions may not be unique. We
fix $u_0 \in E_\alpha$ and set
\begin{equation}  \label{eq:max_time}
	\bar{T}   :=
	\sup\{T > 0; \exists \text{ solution } u \in C^1( (0,T]; E_0) \cap C( (0,T]; E_1) \cap C([0,T];E_{\alpha}) \text{ of } 
	\eqref{eq:QP} \text{ with $u(0) = u_0$} \}
\end{equation}
and prove that the following holds true.

\medskip

\begin{theorem} \label{th:max_ex_time}
Let $0 < \beta < \alpha \leq 1$, let  $u_0 \in E_\alpha$, let
$\bar T$ be defined by \eqref{eq:max_time} and assume that $\bar T < \infty$. 
Further  let $\gamma \in (\beta, \alpha]$ and $R > 0$.
Then there exists a time $T_R(\gamma) \in (0, \bar T)$ with the following property.
If $T \ge T_R(\gamma)$ and if
\begin{equation*}
	u \in C((0,T];E_1) \cap C^1((0,T];E_0) \cap C([0,T];E_{\alpha})
\end{equation*}
is  a solution of \eqref{eq:QP} with $u(0) = u_0$, then
\begin{equation} 
	||u(t)||_{E_\gamma} > R  \quad  \text{for all $t \ge  T_R(\gamma)$}.
\end{equation}
\end{theorem}

\medskip

\begin{proof} Fix $\gamma \in (\beta, \alpha]$. Let $T_0 = T_0(\gamma, R)$ be the existence time in 
Theorem~\ref{th:existence_quasilinear_abstract} , with $\alpha$ replaced by $\gamma$ and
$R'$ replaced by $R$. We claim that the assertion of the theorem holds with
\begin{equation*}
	T_R(\gamma) = \bar T - \frac{T_0}{2}.
\end{equation*}
Indeed, assume that there exists a solution $u \in C([0,T];E_{\alpha})$ with 
\begin{equation*}
	\| u(\bar t)\|_{E_\gamma} \le R \quad \text{and} \quad T \ge \bar t \ge \bar T - \frac{T_0}{2}.
\end{equation*}
Then by Theorem~\ref{th:existence_quasilinear_abstract} there exists a solution 
\begin{equation*}
	U \in C([0, T_0]; E_\gamma) \cap C^{\frac{\gamma-\beta}{2}}([0, T_0]; E_\beta)
\end{equation*}
with initial value $U(0) = u(\bar t)$. Here we used Theorem~\ref{th:existence_quasilinear_abstract} 
with $\alpha = \gamma$ and $\alpha' = \frac{\beta + \gamma}{2}$.
Thus 
\begin{equation*}
	\mathcal A \circ U \in C^\rho([0,T_0]; \mathcal H(E_1;E_0)) 
	\quad 
	\text{and} 
	\quad \mathcal F \circ U \in C^{\rho}([0, T_0]; E_0) 
\end{equation*}
with $\rho = \mu \frac{\gamma-\beta}{2}$. Moreover $U(0) = u(\bar t) \in E_1$. 
Therefore the linear theory gives $U \in C^1([0, T_0];E_0) \cap C([0, T_0]; E_1)$.
Now define
\begin{equation*}
	\tilde{u}(t) =
	\begin{cases}
		u(t), & 0 \leq t < \bar t
		\\
		U(t - \bar t), & \bar t \leq t \leq \bar t + T_0.
	\end{cases}
\end{equation*}
Then $\tilde u \in C([0, \bar t + T_0]; E_\alpha) \cap  C((0, \bar t + T_0]; E_1)$.
The equation 
\begin{equation*}
	\dot{\tilde{u}} + \Acal(\tilde{u})\, \tilde{u} = \Fcal(\tilde{u})
\end{equation*}
holds in $(0, \bar t)$ and in $(\bar t, \bar t + T_0)$. Since $\tilde u \in   C((0, \bar t + T_0]; E_1)$ is follows that 
$\dot{\tilde{u}}$ can be uniquely continued at $t = \bar t$ to a continuous function with values in $E_0$. 
Indeed, thanks to the continuity of $\Acal$ and $\Fcal$ we have for $t > \bar{t}$
\begin{align*}
	0 
	&=
	\lim_{t \searrow \bar{t}} \partial_t^+ U(t-\bar{t}) + \Acal(U(t-\bar{t}))\, U(t-\bar{t}) - \Fcal(U(t - \bar{t}))
	\\
	&=
	\partial_t^+ U(0) + \Acal(U(0))\, U(0) - \Fcal(U(0))
	\\
	&=
	\partial_t^+ \tilde{u}(\bar{t}) + \Acal(\tilde{u}(\bar{t}))\, \tilde{u}(\bar{t}) - \Fcal(\tilde{u}(\bar{t}))
	\quad
	\text{in } E_0
\end{align*}
and for $t < \bar{t}$
\begin{align*}
	0
	&=
	\lim_{t \nearrow \bar{t}} \partial_t^- u(t) + \Acal(u(t))\, u(t) - \Fcal(u(t))
	\\
	&=
	\partial_t^- \tilde{u}(\bar{t}) + \Acal(\tilde{u}(\bar{t}))\, \tilde{u}(\bar{t}) - \Fcal(\tilde{u}(\bar{t}))
	\quad
	\text{in } E_0.
\end{align*}
Thus $\tilde u \in C([0, \bar t + T_0]; E_\alpha) \cap  C((0, \bar t + T_0]; E_1) \cap C^1((0, \bar t + T_0; E_0)$ 
and $\tilde u$ is a solution of $\dot{\tilde{u}} + \Acal(\tilde{u})\, \tilde{u} = \Fcal(\tilde{u})$
on $(0, \bar t + T_0)$. Since $\bar t + T_0 \ge \bar T + \frac{T_0}{2} > \bar T$
this contradicts the definition of $\bar T$. 
\end{proof}

\medskip


\section{Existence of solutions to the non-Newtonian thin-film equation} \label{sec:application_existence}

In this section we apply the abstract existence result Theorem \ref{th:existence_quasilinear_abstract}
to the non-Newtonian thin-film equation
\begin{equation} \label{eq:thin-film}
	\begin{cases}
		u_t +  a \bigl(u^3 \bigl[1 + |b u u_{xxx}|^{\alpha-1}\bigr] u_{xxx}\bigr)_x
		&=\quad
		0,
		\quad
		t > 0, \, x \in \Omega
		\\
		u_x = u_{xxx}
		&=\quad
		0,
		\quad
		t > 0, \, x \in \partial\Omega
		\\
		u(0,\cdot)
		&=\quad
		u_0(\cdot),
		\quad
		x \in \Omega.
	\end{cases}
\end{equation}

We first introduce some notation.
Using the identity 
\begin{equation*}
	 \bigl(u^3 \bigl[1 + |b u u_{xxx}|^{\alpha-1}\bigr] u_{xxx}\bigr)_x
	 =
	 u^3 \bigl(1 + \alpha |b u u_{xxx}|^{\alpha-1}\bigr) u_{xxxx} 
	 + 
	 3u^2 \bigl(1 + |\tilde{b} u u_{xxx}|^{\alpha-1}\bigr) u_x u_{xxx},
\end{equation*}
where $\tilde{b} = \sigma/\tau_{\ast}$, we may rewrite \eqref{eq:thin-film}
in the following way in non-divergence form:
\begin{equation} \label{eq:thin-film_A}
	\begin{cases}
		u_t + A(u,u_x,u_{xx},u_{xxx})  u_{xxxx}
		&=\quad
		F(u,u_x,u_{xx},u_{xxx}),
		\quad
		t > 0, \, x \in \Omega
		\\
		u_x = u_{xxx}
		&=\quad
		0,
		\quad
		t > 0, \, x \in \partial\Omega
		\\
		u(0,\cdot)
		&=\quad
		u_0(\cdot),
		\quad
		x \in \Omega,
	\end{cases}
\end{equation}
where
\begin{equation} \label{eq:A}
	A : (0,\infty)\times\R^3 \longrightarrow (0,\infty),
	\quad
	A(z_0,z_1,z_2,z_3) = a z_0^3 \bigl(1 + \alpha |b z_0 z_3|^{\alpha-1}\bigr)
\end{equation}
and
\begin{equation}\label{eq:F}
	F : \R^4 \longrightarrow \R,
	\quad
	F(z_0,z_1,z_2,z_3) = - 3 z_0^2 \bigl(1 + |\tilde{b} z_0 z_3|^{\alpha-1}\bigr) z_1 z_3.
\end{equation}
are $(\alpha-1)$-H\"older continuous.
Moreover, given a function $v : [0,T]\times \Omega \to \R$ we write
\begin{equation*}
	V(t,x) = (v,v_x,v_{xx},v_{xxx})
\end{equation*}
and we use the abbreviations
\begin{equation*}
	A_V(t) = (A\circ V)(t)
	\quad
	\text{and}
	\quad
	F_V(t) = (F\circ V)(t).
\end{equation*}


Constructing solutions of \eqref{eq:thin-film}, respectively \eqref{eq:thin-film_A}, 
naturally involves the following two challenges.
First, to be able to apply the abstract existence result Theorem \ref{th:existence_quasilinear_abstract}
we have to reformulate \eqref{eq:thin-film_A} as an abstract quasilinear Cauchy problem. In other words
we have to choose a suitable Banach space $E_0$ in which we study the problem and we have to define
the differential operator $\Acal$ properly. This means in particular that we have to define $\Acal$ such that
its domain
\begin{equation*}
	D(\Acal) = \{u \in E_0; \Acal(v)u \in E_0\ \forall v \in E_{\beta};\ u_x = u_{xxx} = 0 \text{ on } \partial\Omega\}
\end{equation*}
incorporates the first- and third-order Neumann boundary conditions. Of course we need that $\Acal$ generates an 
analytic semigroup on $E_0$ and that $\Acal$ and $\Fcal$ satisfy the required regularity properties. Moreover,
$(E_1,E_0)$ has to be a densely and compactly injected Banach couple.


The second challenge we have to deal with is that the non-Newtonian thin-film equation is reasonable 
for positive film heights $u$ only. Hence, in order to apply Theorem \ref{th:existence_quasilinear_abstract}
we extend problem \eqref{eq:thin-film_A} in a way such that for positive initial data solutions of the extended 
problem coincide for a short time with solutions of the original problem.

To tackle the latter challenge we extend the coefficient map $A$ to a globally defined locally H\"older 
continuous function as follows. For $v^+ = \max(v,0)$ we first introduce the map
\begin{equation*}
	\bar{A}: \R^4 \longrightarrow [0,\infty),
	\quad
	\bar{A}(z_0,z_1,z_2,z_3) = A(z_0^+,z_1,z_2,z_3).
\end{equation*}
Note that the function $v \mapsto v^+$ is locally Lipschitz-continuous and hence we still have
$\bar{A} \in C^{\alpha-1}_{\text{loc}}\bigl(\R^4\bigr)$.
Finally, let $\eps > 0$ be given. To ensure parabolicity of the coefficient map we set
\begin{equation*}
	\bar{A}_{\eps} : \R^4 \longrightarrow (\eps/2,\infty),
	\quad
	\bar{A}_{\eps}(z_0,z_1,z_2,z_3) = \max\bigl(\bar{A}(z_0,z_1,z_2,z_3),\eps/2\bigr).
\end{equation*}
Summarising, we have that the maps $A, \bar{A}, \bar{A}_{\eps}$ and $F$ are locally 
$(\alpha-1)$-H\"older continuous on $\R^4$, in symbols
\begin{equation*}
	\bar{A}_{\eps} \in C^{\alpha-1}_{\text{loc}}\bigl(\R^4;(0,\infty)\bigr),
	F \in C^{\alpha-1}_{\text{loc}}\bigl(\R^4;(0,\infty)\bigr),
\end{equation*}
(and analogously for $A$ and $\bar{A}$). That is,
for all $z, z' \in \R^4$ with $|z|, |z'| \leq R$ they satisfy
\begin{equation} \label{eq:A_F_loc_Holder}
	|\bar{A}_{\eps}(z) - \bar{A}_{\eps}(z')|
	\leq
	C_R |z - z'|^{\alpha - 1}
	\quad
	\text{and}
	\quad
	|F(z) - F(z')|
	\leq
	C_R |z - z'|^{\alpha - 1}
\end{equation}
(and analogously for $A$ and $\bar{A}$).

As above we finally introduce the notation
\begin{equation*}
	\bar{A}_{\eps,V}(t) = (\bar{A}_{\eps}\circ V)(t).
\end{equation*}
The corresponding global version of \eqref{eq:thin-film_A} then reads
\begin{equation} \label{eq:thin-film_A_global}
	\begin{cases}
		u_t + \bar{A}_{\eps}(u,u_x,u_{xx},u_{xxx})\, u_{xxxx}
		&=\quad
		F(u,u_x,u_{xx},u_{xxx}),
		\quad
		t > 0, \, x \in \Omega
		\\
		u_x = u_{xxx}
		&=\quad
		0,
		\quad
		t > 0, \, x \in \partial\Omega
		\\
		u(0,\cdot)
		&=\quad
		u_0(\cdot),
		\quad
		x \in \Omega.
	\end{cases}
\end{equation}

The task of setting up an appropriate framework for the abstract Cauchy problem in terms
of function spaces is addressed in the following two subsections. 


\subsection{Solutions to \eqref{eq:thin-film} in fractional Sobolev spaces} \label{sec:Sobolev}

In this section we study the problem of existence of solutions to \eqref{eq:thin-film},
respectively \eqref{eq:thin-film_A}, which are H\"older continuous in time and take values 
in Sobolev spaces of fractional order. Note that we consider only the case in which
$\Omega \subset \R$ is a bounded interval.

For $k \in \N$ and $p \in [1,\infty)$ we denote by $W^k_p(\Omega)$ the usual Sobolev 
spaces with norm
\begin{equation*}
	||v||_{W^k_p(\Omega)} = \left(\sum_{j=0}^k ||\partial^{j} v||_{L_p(\Omega)}^p\right)^{1/p}.
\end{equation*}
We then put
\begin{equation*}
	[v]_{W^s_p(\Omega)}
	=
	\int_{\Omega} \int_{\Omega} \frac{|v(x) - v(z)|^p}{|x-z|^{1+sp}}\, dx\, dz,
	\quad
	1 \leq p < \infty,\ 0 < s < 1,
\end{equation*}
and define the \textbf{Sobolev--Slobodeckii} or \textbf{fractional Sobolev spaces} by
\begin{equation*}
	W^s_p(\Omega) =
	\left\{v \in W^{[s]}_p(\Omega); ||v||_{W^s_p(\Omega)} < \infty\right\},
	\quad
	1\leq p < \infty,\ s \in \R_+\setminus\N,
\end{equation*}
where
\begin{equation*}
	||v||_{W^s_p(\Omega)} = \left(||v||_{W^{[s]}_p(\Omega)}^p + [\partial^{[s]}v]_{W^{s-[s]}_p(\Omega)}^p\right)^{1/p},
	\quad
	1\leq p < \infty,\ s \in \R_+\setminus\N.
\end{equation*}
Here $[s]$ denotes the largest integer smaller than or equal to $s$.

We now recall some important properties of these spaces which are necessary to guarantee that
we are in the setting of Theorem \ref{th:existence_quasilinear_abstract}. It is well-known that 
for $-\infty < s_0 < s_1 < \infty$ and $0 < \rho < 1$ the space $W^s_p(\Omega), 
s = (1 - \rho) s_0 + \rho s_1$, is the complex interpolation space between $W^{s_1}_p(\Omega)$ and 
$W^{s_0}_p(\Omega)$, in symbols
\begin{equation*}
	W^s_p(\Omega) = [W^{s_0}_p(\Omega),W^{s_1}_p(\Omega)]_{\rho}.
\end{equation*}

In order to take the (Neumann) boundary conditions of \eqref{eq:thin-film}, respectively
\eqref{eq:thin-film_A}, into account we further introduce the Banach spaces
\begin{equation*}
	W^{4\rho}_{p,B}(\Omega)
	=
	\begin{cases}
		\left\{v \in W^{4\rho}_p(\Omega); 
		v_x = v_{xxx} = 0 \text{ on } \partial\Omega\right\}, 
		& 3 + \frac{1}{p} < 4\rho \leq 4
		\\[1ex]
		\left\{v \in W^{4\rho}_p(\Omega); 
		v_x = 0 \text{ on } \partial\Omega\right\}, 
		& 1+ \frac{1}{p} < 4\rho \leq 3 + \frac{1}{p}
		\\[1ex] 
		W^{4\rho}_p(\Omega), & 0 \leq 4\rho \leq 1+ \frac{1}{p}.
	\end{cases}
\end{equation*}
For $4\rho \in (0,4)\setminus\{1+1/p,3+1/p\}$ the spaces $W^{4\rho}_{p,B}(\Omega)$ are closed
linear subspaces of $W^{4\rho}_{p}(\Omega)$ and satisfy the interpolation property
\begin{equation*}
	W^{4\rho}_{p,B} = \bigl(L_p,W^4_{p,B}(\Omega)\bigr)_{\rho,p},
	\quad
	1 < p < \infty.
\end{equation*}


We can now apply the abstract existence result Theorem \ref{th:existence_quasilinear_abstract}
to the non-Newtonian thin-film equation \eqref{eq:thin-film_A}. More precisely, we prove the following theorem
on the existence of solutions in Sobolev spaces of fractional order.


\begin{theorem} \label{th:existence_Sobolev}
Let $p \in (1,\infty)$ and $1/p < s < r < 1$. Moreover, let $\sigma = \frac{3+s}{4}$ and 
$\rho = \frac{3+r}{4}$. Then, given an initial film height 
$u_0 \in W^{4\rho}_{p,B}(\Omega)$ such that $u_0(x) > 0\,$ for all $x \in \bar{\Omega}$, 
for each $\alpha \in (1,2)$ there exists a positive $T > 0$ and
a solution $u$ of \eqref{eq:thin-film} on $[0,T]$ in the sense that
\begin{equation*}
	u \in C\bigl([0,T],W^{4\rho}_{p,B}(\Omega)\bigr)
	\cap
	C^{\rho}\bigl([0,T],L_p(\Omega)\bigr)
	\cap
	C\bigl((0,T],W^4_{p,B}(\Omega)\bigr)
	\cap
	C^1\bigl((0,T],L_p(\Omega)\bigr)
\end{equation*}
and
\begin{equation*}
	u(t,x) > 0,
	\quad
	(t,x) \in [0,T]\times\bar{\Omega}.
\end{equation*}
\end{theorem}


In order to prove Theorem \ref{th:existence_Sobolev} we have to verify that the conditions 
of the abstract result Theorem \ref{th:existence_quasilinear_abstract} are satisfied. 
To this end we make the choice
\begin{equation*}
	E_0 = L_p(\Omega),
	\quad
	E_1 = W^4_{p,B}(\Omega).
\end{equation*}
For this choice it is well-known that
\begin{equation*}
	W^4_{p,B}(\Omega) \xhookrightarrow{d} L_p(\Omega)
	\quad
	\text{and}
	\quad
	W^4_{p,B}(\Omega) \xhookrightarrow{c} L_p(\Omega),
	\quad
	1 < p < \infty.
\end{equation*}
Denoting by $E_{\rho}=W^{4\rho}_{p,B}(\Omega), 
4\rho \in (0,4)\setminus\{1+1/p,3+1/p\}$, 
the respective complex interpolation spaces,
this implies (see for instance \cite[Thm. I.2.11.1]{A:1995}) that also
\begin{equation*}
	W^{4\rho}_{p,B}(\Omega) \xhookrightarrow{d} 
	W^{4\sigma}_{p,B}(\Omega) 
	\quad
	\text{and}
	\quad
	W^{4\rho}_{p,B}(\Omega) \xhookrightarrow{c} 
	W^{4\sigma}_{p,B}(\Omega),
	\quad
	0 \leq \sigma < \rho \leq 1.
\end{equation*}


Based on that we view the evolution equation $\eqref{eq:thin-film_A}_1$ in non-divergence form
as an abstract quasilinear Cauchy problem in the following way. Let $p \in (1,\infty)$ and $s > 1/p$.
For $v \in W^{4\sigma}_{p,B}(\Omega)$ with $\sigma = \frac{3+s}{4}$ such that $v(x) > 0$ for all
$x \in \bar{\Omega}$ we associate to \eqref{eq:thin-film_A}
the linear differential operator 
\begin{equation} \label{eq:operator}
	\Acal(v(t)) \in \Lcal\bigl(W^4_{p,B}(\Omega);L_p(\Omega)\bigr),
	\quad
	\Acal(v(t)) u := A_V(t) \partial_x^4 u
\end{equation}
of fourth order. Then, with
\begin{equation} \label{eq:rhs}
	\Fcal(v(t)) =
	-3a \bigl(v^2 v_x v_{xxx} + \tilde{b}^{\alpha-1} v^{\alpha+1} v_x |v_{xxx}|^{\alpha-1} v_{xxx}\bigr)
\end{equation}
we rewrite \eqref{eq:thin-film_A} as
\begin{equation} \label{eq:CP_Sobolev}
	\begin{cases}
		\dot{u} + \Acal(u) u &=\quad \Fcal(u),
		\quad t > 0
		\\
		u(0) &=\quad u_0.
	\end{cases}
\end{equation}


Similarly we rewrite the extended problem \eqref{eq:thin-film_A_global} as
\begin{equation} \label{eq:CP_Sobolev_global}
	\begin{cases}
		\dot{u} + \bar{\Acal}_{\eps}(u) u &=\quad \Fcal(u),
		\quad t > 0
		\\
		u(0) &=\quad u_0,
	\end{cases}
\end{equation}
where
\begin{equation} \label{eq:operator_global}
	\bar{\Acal}_{\eps}(v(t)) \in \Lcal\bigl(W^4_{p,B}(\Omega);L_p(\Omega)\bigr),
	\quad
	\bar{\Acal}_{\eps}(v(t)) u := \bar{A}_{\eps,V}(t) \partial_x^4 u
\end{equation}


In the following lemmas we study the relevant regularity properties 
of the differential operator $\bar{\Acal}_{\eps}$ and the right-hand side $\Fcal$, introduced 
\eqref{eq:operator_global}, respectively \eqref{eq:rhs}.


\begin{lemma} \label{lem:A_F_Holder}
Given $p \in (1,\infty)$ and $1/p < s < 1$, let $\sigma = \frac{3+s}{4}$. Then
for all flow behaviour exponents $\alpha \in (1,2)$ the mappings
\begin{equation*}
	\bar{\Acal}_{\eps} : W^{4\sigma}_{p,B}(\Omega) \to \Lcal\bigl(W^4_{p,B}(\Omega);L_p(\Omega)\bigr)
	\quad
	\text{and}
	\quad
	\Fcal : W^{4\sigma}_{p,B}(\Omega) \longrightarrow L_p(\Omega),
\end{equation*}
are $(\alpha-1)$-H\"older continuous on bounded balls
in the sense that
\begin{equation*}
	||\bar{\Acal}_{\eps}(v) - \bar{\Acal}_{\eps}(w)||_{\Lcal(W^4_{p,B}(\Omega),L_p(\Omega))}
	\leq
	C_R
	||v - w||_{W^{4\sigma}_{p,B}(\Omega)}^{\alpha-1}
	\quad
	\text{and}
	\quad
	||\Fcal(v) - \Fcal(w)||_{L_p(\Omega)} 
	\leq 
	C_R 
	||v - w||_{W^{4\sigma}_{p,B}(\Omega)}^{\alpha-1}
\end{equation*}
for all $v, w \in W^{4\sigma}_{p,B}(\Omega)$ with 
$||v||_{W^{4\sigma}_{p,B}(\Omega)},||w||_{W^{4\sigma}_{p,B}(\Omega)} \leq R$. 
\end{lemma}

\medskip

\begin{proof}
This follows from \eqref{eq:A_F_loc_Holder}.
\end{proof}

\medskip

We can now prove the main result of this section.

\medskip

\begin{proof}[\textbf{Proof of Theorem \ref{th:existence_Sobolev}}]
(i) Existence. Let $p \in (1,\infty)$ and $1/p < s < r < 1$. Then put $\sigma = \frac{3+s}{4}$ 
and $\rho = \frac{3+r}{4}$.
Suppose that $u_0 \in W^{4\rho}_{p,B}$ such that 
\begin{equation*}
	||u_0||_{W^{4\rho}_{p,B}(\Omega)} < R'
	\quad
	\text{and}
	\quad
	u_0(x) \geq 2\Bigl(\frac{\eps}{2a}\Bigr)^{1/3} > 0
	\quad 
	\forall x \in \bar{\Omega},
\end{equation*}
where $\eps, R' > 0$ are fixed.

We first apply Theorem \ref{th:existence_quasilinear_abstract} to show that problem
\eqref{eq:CP_Sobolev_global} possesses a solution for some time $T_{\eps} > 0$. 
To this end note that in view of Lemma \ref{lem:A_F_Holder}  we
have the required H\"older continuity of the operator $\bar{\Acal}_{\eps}$ and the right-hand side;
\begin{equation} \label{eq:Sobolev_reg}
	\begin{split}
	\bar{\Acal}_{\eps} &: W^{4\sigma}_{p,B}(\Omega) \longrightarrow \Lcal\bigl(W^4_{p,B}(\Omega);L_p(\Omega)\bigr)
	\\
	\Fcal &: W^{4\sigma}_{p,B}(\Omega) \longrightarrow L_p(\Omega).
	\end{split}
\end{equation}
Note that here we used that $\bar{\Acal}_{\eps}$ is the composition of $\Acal$ with two Lipschitz continuous maps.
Moreover, recall that the choice of $\sigma= \frac{3+s}{p}$ with $s > 1/p$ implies that 
$W^{4\sigma}_{p,B}(\Omega) \hookrightarrow C(\bar{\Omega})$, whence for 
$v \in W^{4\sigma}_{p,B}(\Omega)$
\begin{equation*}
	A_V = a v^3\bigl(1+\alpha |b v v_{xxx}|^{\alpha-1}\bigr) \in C(\bar{\Omega})
\end{equation*}
and thus finally 
\begin{equation*}
	\bar{A}_{\eps} = \max\bigl(A((v)_+,v_x,v_{xx},v_{xxx}),\eps/2\bigr) \in C(\bar{\Omega}).
\end{equation*}
In addition the principal symbol $a_{\eps}(x,\xi)$ of the operator $\Acal_{\eps}(v)$ 
satisfies the uniform Legendre--Hadamard condition
\begin{equation*}
	\text{Re}\bigl(a_{\eps}(x,\xi) \eta | \eta\bigr) \geq \frac{\eps}{2} (i \xi)^4 \eta^2 > 0
\end{equation*}
for $(x,\xi) \in \bar{\Omega}\times\{-1,1\} \text{ and } \eta \in \R\setminus\{0\}$. 
Thus $-\bar{\Acal}_{\eps}(v)$ with the given boundary conditions is normally elliptic in the sense of
\cite[Example 4.3(d)]{A:1993}. Thanks to \cite[Thm. 4.1 \& Rem. 4.2(b)]{A:1993} we conclude that 
\begin{equation} \label{eq:Sobolev_Hcal}
	\bar{\Acal}_{\eps}(v) \in \Hcal\bigl(W^4_{p,B}(\Omega),L_p(\Omega)\bigr)
\end{equation}
i.e. $-\bar{\Acal}_{\eps}(v)$ generates an analytic semigroup on $L_p(\Omega)$. In virtue of 
\eqref{eq:Sobolev_reg} and \eqref{eq:Sobolev_Hcal} we may eventually apply Theorem 
\ref{th:existence_quasilinear_abstract} to conclude that there exists a positive time $T_{\eps}$ and
a solution 
\begin{equation*}
	u_{\eps} \in C\bigl([0,T];W^{4\rho}_{p,B}(\Omega)\bigr) 
	\cap
	C^{\nu}\bigl([0,T];W^{4\sigma}_{p,B}(\Omega)\bigr), 
\end{equation*}
with $\nu \in (0,\rho-\sigma)$, to the extended problem \eqref{eq:CP_Sobolev_global}.
\medskip

(ii) Positivity. 
As above we denote by $\bar{u}_0$ the constant extension of $u_0$ on $[0,T]$. 
Now if $u_0(x) \geq 2\bigl(\frac{\eps}{2a}\bigr)^{1/3}$ for all $x \in \bar{\Omega}$ we find that
\begin{equation*}
	\min_{t \in [0,T]} u_{\eps}(t,x)
	\geq
	2\Bigl(\frac{\eps}{2a}\Bigr)^{1/3} - CT^{\nu}
\end{equation*}
for all $x \in \bar{\Omega}$. Hence 
\begin{equation} \label{eq:positivity}
	u_{\eps}(t,x) > \Bigl(\frac{\eps}{2a}\Bigr)^{1/3},
	\quad
	(t,x) \in [0,T]\times\bar{\Omega}
\end{equation}
for $T < \bigl(\frac{1}{C}(\frac{\eps}{2a})^{1/3}\bigr)^{1/\nu}$. 
\medskip

(iii) It remains to show that the solution $u_{\eps}$ is -- at least for a short time -- also a solution to
\eqref{eq:CP_Sobolev}. Indeed, by \eqref{eq:positivity} we obtain
\begin{equation*}
	\min_{t \in [0,T]} A_{U_{\eps}}(t) 
	\geq
	\min_{t \in [0,T]}
	a\, u_{\eps}(t)^3
	\geq
	\frac{\eps}{2}
\end{equation*}
for $T < \bigl(\frac{1}{C}(\frac{\eps}{2a})^{1/3}\bigr)^{1/\nu}$. This implies that 
there exists a positive time $T^{\ast}$ such that
\begin{equation*}
	A_{U_{\eps}}(t) = \bar{A}_{\eps,U_{\eps}}(t),
	\quad
	t \in [0,T^{\ast}],
\end{equation*}
and hence $u_{\eps}$ does also solve the original problem \eqref{eq:CP_Sobolev} on $[0,T^{\ast}]$.
This completes the proof.
\end{proof}

\medskip

\begin{remark}
It is worthwhile to discuss again the qualitative differences originating from the 
different values for the flow behaviour exponent $\alpha$.
Note that for $\alpha \geq 2$ Lemma \ref{lem:A_F_Holder} can be
improved to Lipschitz continuity in the appropriate norms. Indeed, recall that $W^s_p(\Omega)$
is a Banach algebra and that the Nemitskii operator induced by the function $f(z) = |z|^{\alpha-1}$
acts on $W^s_p(\Omega)$ for $\alpha \geq 2$, see for instance \cite[Thm. 4.6.4/2]{RS:1996}, respectively 
\cite[Thm. 5.4.3/1]{RS:1996}. Together with the inequality
\begin{equation*}
	|\, |z|^{\beta} - |y|^{\beta}| \leq C_{\beta} \bigl(|z|^{\beta-1} + |y|^{\beta-1}\bigr) |z-y|,
	\quad
	\beta \geq 1,
\end{equation*}
this makes the following calculation possible. For $p \in (1,\infty)$ and $s > 1/p$ we have
\begin{align*}
	& 
	||(v^{\alpha+2} |v_{xxx}|^{\alpha-1} - w^{\alpha+2} |w_{xxx}|^{\alpha-1}) u_{xxxx}||_{L_p(\Omega)} 
	\\
	& \leq
	C
	\Bigl(
	||(v^{\alpha+2} - w^{\alpha+2}) |v_{xxx}|^{\alpha-1}||_{L_p(\Omega)} 
	+
	||w^{\alpha+2} (|v_{xxx}|^{\alpha-1} - |w_{xxx}|^{\alpha-1})||_{L_p(\Omega)}
	\Bigr) ||u||_{W^4_p(\Omega)}
	\\
	& \leq
	C_{\alpha} 
	\Bigl(||\,|v_{xxx}|^{\alpha-1}||_{W^s_p(\Omega)} ||v^{\alpha+2} - w^{\alpha+2}||_{L_{\infty}(\Omega)}
	+
	||w^{\alpha+2}||_{W^s_p(\Omega)} ||\,|v_{xxx}|^{\alpha-1} - |w_{xxx}|^{\alpha-1}||_{L_{\infty}(\Omega)}
	\Bigr) ||u||_{W^4_p(\Omega)}
	\\
	& \leq
	C_{\alpha} 
	\Bigl(||v - w||_{W^{3+s}_p(\Omega)}
	+
	\bigl[||\,|v_{xxx}|^{\alpha-2}||_{L_{\infty}(\Omega)} + ||\,|w_{xxx}|^{\alpha-2}||_{L_{\infty}(\Omega)}\bigr]
	||v_{xxx} - w_{xxx}||_{W^s_p(\Omega)}
	\Bigr) ||u||_{W^4_p(\Omega)}
	\\
	&\leq
	C_{\alpha} 
	||v - w||_{W^{3+s}_p(\Omega)}.
\end{align*}
This means that for $\alpha \geq 2$ one can even prove that 
$\Acal \in \text{Lip}\bigl(W^{4\sigma}_p(\Omega);\Hcal(W^4_p(\Omega),L_p)\bigr)$,
where $\sigma = (3+s)/4$. A similar calculation shows that $\Fcal \in \text{Lip}\bigl(W^{4\sigma}_p(\Omega);L_p(\Omega)\bigr)$.
Hence for $\alpha \geq 2$ we are in the regime of \cite[Thm. 12.1]{A:1993} which gives existence and uniqueness of solutions
to \eqref{eq:thin-film} in the sense of Theorem \ref{th:existence_Sobolev}.
\end{remark}

\medskip


\subsection{Solutions of \eqref{eq:thin-film} in (little) H\"older spaces} \label{sec:Holder}


This section is devoted to the existence of classical solutions to the non-Newtonian thin-film
equation \eqref{eq:thin-film}, respectively \eqref{eq:thin-film_A}. More precisely we apply our
abstract Theorem \ref{th:existence_quasilinear_abstract} in the setting of (little) H\"older spaces.
Note again that we study the one-dimensional thin-film equation.
\medskip


As in Section \ref{sec:Sobolev} we start by introducing the relevant notation and function spaces.
Let $\Omega \subset\R$ be an open and bounded interval. For $k \in \N$ and $\rho \in (0,1)$ we 
define the usual H\"older spaces by
\begin{equation*}
	C^{\rho}(\bar{\Omega})
	=
	\Bigl\{v \in C(\bar{\Omega});\ 
	[v]_{C^{\rho}(\bar{\Omega})}
	=
	\sup_{x,z\in \bar{\Omega}, x\neq z} \frac{|v(x) - v(z)|}{|x-z|^{\rho}} < \infty\Bigr\}
	\quad
	\text{with}
	\quad
	||v||_{C^{\rho}(\bar{\Omega})} = ||v||_{C(\bar{\Omega})} + [v]_{C^{\rho}(\bar{\Omega})}
\end{equation*} 
and
\begin{equation*}
	C^{k+\rho}(\bar{\Omega})
	=
	\Bigl\{v \in C^k(\bar{\Omega});\ 
	[v^{(k)}]_{C^{\rho}(\bar{\Omega})} < \infty\Bigr\}
	\quad
	\text{with}
	\quad
	||v||_{C^{k+\rho}(\bar{\Omega})} = ||v||_{C^k(\bar{\Omega})} + [v^{(k)}]_{C^{\rho}(\bar{\Omega})}.
\end{equation*} 
We further introduce the so-called little-H\"older spaces
\begin{equation*}
	h^{\rho}(\bar{\Omega})
	=
	\Bigl\{v \in C^{\rho}(\bar{\Omega});\ 
	\lim_{\eps \to 0} \sup_{x, z \in \bar{\Omega}; 0<|x-z|<\eps} \frac{|v(x) - v(z)|}{|x-z|^{\rho}} = 0\Bigr\}
\end{equation*} 
and
\begin{equation*}
	h^{k+\rho}(\bar{\Omega})
	=
	\Bigl\{v \in C^{k+\rho}(\bar{\Omega});\ 
	\lim_{\eps \to 0} \sup_{x, z \in \bar{\Omega}; 0<|x-z|<\eps} \frac{|v^{(k)}(x) - v^{(k)}(z)|}{|x-z|^{\rho}} = 0\Bigr\}
\end{equation*} 
We recall some important properties of these spaces.


The space $h^{\rho}(\bar{\Omega})$ is a closed subspace of $C^{\rho}(\bar{\Omega})$ and hence
a Banach space.

If $0 < \sigma < 1$, then $h^{\sigma}(\bar{\Omega})$ is the closure of $C^{\rho}(\bar{\Omega})$ 
in $C^{\sigma}(\bar{\Omega})$ for all $\rho \in (\sigma,\infty]$.

Furthermore for $0\leq s_0 < s_1$ and $0 < \rho < 1$ the space $h^s(\bar{\Omega}), 
s = (1-\rho) s_0 + \rho s_1$ is the real interpolation space between $C^{s_1}(\bar{\Omega})$
and $C^{s_0}(\bar{\Omega})$, in symbols
\begin{equation*}
	h^s(\bar{\Omega})
	=
	\bigl(C^{s_0}(\bar{\Omega}),C^{s_1}(\bar{\Omega})\bigr)_{\rho},
	\quad
	s \notin \N,\ 0\leq s_0 < s_1.
\end{equation*}


In order to take the first and third order Neumann boundary conditions of problem \eqref{eq:thin-film}
into account we further introduce for $\rho \in (0,1]$ the spaces
\begin{equation*}
	h^{4\rho}_B(\bar{\Omega})
	=
	\begin{cases}
		\{v \in h^{4\rho}(\bar{\Omega});\ v_x = v_{xxx} = 0 \text{ on } \partial\Omega\},
		& 
		3 \leq 4\rho \leq 4
		\\
		\{v \in h^{4\rho}(\bar{\Omega});\ v_x = 0 \text{ on } \partial\Omega\},
		& 
		1 \leq 4\rho < 3
		\\
		h^{4\rho}(\bar{\Omega}),
		&
		0 < 4\rho < 1.
	\end{cases}
\end{equation*}
For $4\rho \in (0,4), 4\rho \notin \N$, the spaces $h^{4\rho}_B(\bar{\Omega})$ 
are closed linear subspaces of $h^{4\rho}(\bar{\Omega})$. Thanks to \cite[Thm. 2.3]{AT:1987} 
they may be characterised as the real interpolation spaces between 
$C^4_B(\bar{\Omega})$ and $C(\bar{\Omega})$;
\begin{equation*}
	h^{4\rho}_B(\bar{\Omega})
	=
	\bigl(C^4_B(\bar{\Omega}),C(\bar{\Omega})\bigr)_{\rho},
	\quad
	4\rho \notin \N.
\end{equation*}

\medskip

The main result of this section may now be formulated as follows.

\medskip

\begin{theorem} \label{th:existence_Holder}
Let $3/4 < \sigma < \rho \leq 1$. Then, given an initial film height $u_0 \in h^{4\rho}_B(\bar{\Omega})$
such that $u_0(x) > 0$ for all $x \in \bar{\Omega}$, for each $\alpha > 1$ there exists a positive 
time $T > 0$ and a solution $u$ of \eqref{eq:thin-film} on $[0,T]$ in the sense that
\begin{equation*}
	u \in 
	C\bigl([0,T];h^{4\rho}_B(\bar{\Omega})\bigr)
	\cap
	C^{\rho}\bigl([0,T];C(\bar{\Omega})\bigr)
	\cap
	C\bigl((0,T];C^4_B(\bar{\Omega})\bigr)
	\cap
	C^{1}\bigl((0,T];C(\bar{\Omega})\bigr).
\end{equation*}
If in addition $u_0 \in C^{4}_B(\bar{\Omega})$ then 
\begin{equation*}
	C\bigl([0,T];C^4_B(\bar{\Omega})\bigr)
	\cap
	C^{1}\bigl([0,T];C(\bar{\Omega})\bigr).
\end{equation*}
In any case $u$ satisfies
\begin{equation*}
	u(t,x) > 0,
	\quad
	(t,x) \in [0,T]\times\bar{\Omega}.
\end{equation*}
\end{theorem}

The proof of Theorem \ref{th:existence_Holder} is similar to the one in the setting 
of fractional Sobolev spaces.


We verify that the conditions of the abstract result Theorem 
\ref{th:existence_quasilinear_abstract} are satisfied and identify
\begin{equation*}
	E_0 = C(\bar{\Omega}),
	\quad
	E_1 = C^4_{B}(\bar{\Omega}).
\end{equation*}
For this choice it is well-known that
\begin{equation*}
	C^4_{B}(\bar{\Omega}) \xhookrightarrow{d} C(\bar{\Omega})
	\quad
	\text{and}
	\quad
	C^4_{B}(\bar{\Omega}) \xhookrightarrow{c} C(\bar{\Omega}).
\end{equation*}
Denoting by $E_{\rho}=h^{4\rho}_{p,B}(\bar{\Omega})$ the respective interpolation spaces,
we have (see for instance \cite[Thm. I.2.11.1]{A:1995}) 
\begin{equation*}
	h^{4\rho}_{B}(\bar{\Omega}) \xhookrightarrow{d} 
	h^{4\sigma}_{B}(\bar{\Omega}) 
	\quad
	\text{and}
	\quad
	h^{4\rho}_{B}(\bar{\Omega}) \xhookrightarrow{c} 
	h^{4\sigma}_{B}(\bar{\Omega}),
	\quad
	0 \leq \sigma < \rho \leq 1.
\end{equation*}


As before we view the evolution equation $\eqref{eq:thin-film_A}_1$ in non-divergence form
as an abstract quasilinear Cauchy problem.
For $v \in h^{4\sigma}_{B}(\Omega)$ with $\sigma = \frac{3+s}{4}$ such that $v(x) > 0$ for all
$x \in \bar{\Omega}$ we associate to \eqref{eq:thin-film_A}
the linear differential operator 
\begin{equation} \label{eq:operator_Holder}
	\Acal(v(t)) \in \Lcal\bigl(C^4_{B}(\bar{\Omega});C(\bar{\Omega})\bigr),
	\quad
	\Acal(v(t)) u := A_V(t) \partial_x^4 u
\end{equation}
of fourth order. Then, with
\begin{equation} \label{eq:rhs_Holder}
	\Fcal(v(t)) =
	-3a \bigl(v^2 v_x v_{xxx} + \tilde{b}^{\alpha-1} v^{\alpha+1} v_x |v_{xxx}|^{\alpha-1} v_{xxx}\bigr)
\end{equation}
we rewrite \eqref{eq:thin-film_A} as
\begin{equation} \label{eq:CP_Holder}
	\begin{cases}
		\dot{u} + \Acal(u) u &=\quad \Fcal(u),
		\quad t > 0
		\\
		u(0) &=\quad u_0.
	\end{cases}
\end{equation}

\medskip

As in the proof of Theorem \ref{th:existence_Sobolev} we first solve the extended problem
and then prove that the solution also satisfies the original equations.

\medskip

\begin{proof}[Proof of Theorem \ref{th:existence_Holder}]
As in Lemma \ref{lem:A_F_Holder} one obtains that the right-hand side, considered as a map 
$\Fcal : h^{4\sigma}_B(\bar{\Omega}) \to C(\bar{\Omega})$, and the differential operator
$\bar{\Acal}_{\eps} : h^{4\sigma}_B(\bar{\Omega}) \to \Lcal(C^4_B(\bar{\Omega});C(\bar{\Omega}))$
are H\"older continuous on all balls in the sense that
\begin{equation*}
	||\Fcal(v) - \Fcal(w)||_{C(\bar{\Omega})}
	\leq
	C_R
	||v - w||_{h^{4\sigma}_B(\bar{\Omega})}^{\alpha-1}
	\quad
	\text{and}
	\quad
	||\bar{\Acal}_{\eps}(v) - \bar{\Acal}_{\eps}(w)||_{\Lcal(C^4_B(\bar{\Omega});C(\bar{\Omega}))}
	\leq
	C_R 
	||v - w||_{h^{4\sigma}_B(\bar{\Omega})}^{\alpha-1}
\end{equation*}
for all $v, w \in h^{4\sigma}_B(\bar{\Omega})$ with 
$||v||_{h^{4\sigma}_B(\bar{\Omega}))}, ||w||_{h^{4\sigma}_B(\bar{\Omega})} \leq R$.

From \cite{Stewart:1974,Stewart:1980} we know that
$-\bar{\Acal}_{\eps}(v)$ generates for each $v \in h^{4\sigma}_B(\bar{\Omega})$ an 
analytic semigroup on $C(\bar{\Omega})$, i.e.
\begin{equation*}
	\bar{\Acal}_{\eps}(v) \in \Hcal\bigl(C^4_B(\bar{\Omega});C(\bar{\Omega})\bigr).
\end{equation*}

We obtain the assertion by following the lines of the proof of Theorem \ref{th:existence_Sobolev}.
\end{proof}

\medskip


\section{Uniqueness of solutions to \eqref{eq:thin-film} for flow behaviour exponents $\alpha \in (1,2)$}
\label{sec:uniqueness}

Recall from Sections \ref{sec:Sobolev} and \ref{sec:Holder} that for flow behaviour 
exponents $\alpha \geq 2$ we have Lipschitz continuity of the differential operator $\Acal$ as
well as the right-hand side $\Fcal$. Thus, for $\alpha \geq 2$ we are in the setting of
Eidel'man \cite[Thm. III.4.6.3]{E:1969} and Amann \cite[Thm. 12.1]{A:1993} and obtain 
uniqueness of solutions to \eqref{eq:thin-film} by a contraction argument.

\medskip

For flow behaviour exponents $\alpha \in (1,2)$ we get existence of solutions to
\eqref{eq:thin-film} in fractional Sobolev and little H\"older spaces, respectively, by
compactness of the solution operator for the linear problem, c.f. Theorems 
\ref{th:existence_Sobolev} and \ref{th:existence_Holder}.

\medskip
 
In this section we prove uniqueness of solutions to \eqref{eq:thin-film} by deriving an energy
inequality for which we use the special structure of the equation. More precisely, we extend the approach used in 
\cite{BF:1990} for the Newtonian thin-film equation to prove that for $\alpha \in (1,2)$ two positive 
strong solutions of \eqref{eq:thin-film} coincide if this is the case initially. For this purpose observe
that the energy
\begin{equation*}
	E(u) = \frac{1}{2} \int_{\Omega} |u_x|^2 dx
\end{equation*}
decreases along smooth solutions of \eqref{eq:thin-film}. Indeed, if $u$ is a smooth solution of 
\eqref{eq:thin-film}, then
\begin{equation*}
	\frac{d}{dt} E(u(t))
	=
	-\int_{\Omega} u_{xx} u_t\, dx
	=
	-a \int_{\Omega} u^3 |u_{xxx}|^2 + b^{\alpha-1} u^{\alpha+2} |u_{xxx}|^{\alpha+1}\, dx
\end{equation*}
and hence
\begin{equation} \label{eq:energy_eq}
	\frac12 \int_{\Omega} u_x^2(t) dx 
	+ a \int_0^T \int_{\Omega} u^3 |u_{xxx}|^2 + b^{\alpha-1} u^{\alpha+2} |u_{xxx}|^{\alpha+1} dx dt
	=
	\frac12 \int_{\Omega} (u_0)_x^2 dx.
\end{equation}

\medskip

To justify the energy inequality  \eqref{eq:energy_eq} for solutions in our regularity class and to 
apply a similar argument to the difference of two solutions we use the following fact. 

\medskip

\begin{proposition} \label{pr:dual_pairing_differentiable}  
Suppose that 
$w_1, w_2 \in C( (0,T); W^1_{p',0}(\Omega)) \cap C^1( (0,T), W^{-1}_{p}(\Omega))$. 
Then the map $t \mapsto \langle w_1(t), w_2(t) \rangle$ is differentiable in $(0,T)$ and
\begin{equation*}
	\frac{d}{dt} \int_\Omega w_1 w_2 \, dx  
	= 
	\langle \frac{d}{dt} w_1, w_2 \rangle + \langle \frac{d}{dt} w_2, w_1 \rangle,
\end{equation*}
where $\langle \cdot, \cdot \rangle$ denotes the dual pairing between $W^{-1}_{p}(\Omega)$ and $W^{1}_{p'}(\Omega)$. 
\end{proposition}

\medskip

\begin{proof} 
This follows by writing out the difference quotient and noting that $h^{-1}(w_2(t+h) - w_2(t))$ is 
bounded in $W^{-1}_{p}(\Omega)$ while $w_1(t+h) - w_1(t)$ goes to to zero in $W^{1}_{p',0}(\Omega)$ 
as $h \to 0$.
\end{proof}

\medskip

Proposition \ref{pr:dual_pairing_differentiable} guarantees in particular that the expression 
$\frac{d}{dt} E(u(t))$ is well-defined for solutions $u$ obtained
by Theorem \ref{th:existence_Sobolev} in the fractional Sobolev space setting
or by Theorem \ref{th:existence_Holder} in the little H\"older setting.
This allows us to show the following uniqueness result.

\medskip

\begin{theorem} \label{th:uniqueness}
Let $\alpha > 1$. Let $u$ and $v$ be two positive solutions of \eqref{eq:thin-film} as in Theorem
\ref{th:existence_Sobolev} or Theorem \ref{th:existence_Holder}.
on $[0,T]$, emanating from the same initial value $u_0$, where $u_0(x) > 0$ for all $x \in \bar{\Omega}$. 
Then $u=v$ on $[0,T]$.
\end{theorem}

\medskip

\begin{proof}
We consider only solutions obtained by Theorem \ref{th:existence_Sobolev}. The proof for solutions 
in the sense of Theorem \ref{th:existence_Holder} is the same. By the usual continuation argument 
it suffices to show that there exists a time $0 < T_{\ast} \leq T$ such that $u=v$ on $[0,T_{\ast})$.
Since both $u$ and $v$ are positive as long as they exist there is a $0 < T_0 < T$ and constants 
$c, C > 0$ such that
\begin{equation*}
	0 < c \leq u(t,x), v(t,x) \leq C,
	\quad
	t \in [0,T_0],\  x \in \bar{\Omega}.
\end{equation*}
For all $t \in (0,T_0)$ we may now extend the arguments of ~\cite{BF:1990} in the following way.

We know that $(u-v) \in C((0,T]; W^{4}_{p,B}(\Omega)) \cap C^1((0,T); L_p(\Omega))$.  
Since $W^{4}_{p,B}(\Omega)$ embedds into $C^3(\bar \Omega)$
in particular $(u- v)_x \in C((0,T]; C^2(\bar \Omega))$ and $(u-v)_x = 0$ on $\partial \Omega$. Thus it follows from 
Proposition~\ref{pr:dual_pairing_differentiable} that $t \mapsto \int_\Omega (u_x - v_x)^2 \,dx$ is differentiable in $(0,T)$ and
\begin{equation*}
	\frac{d}{dt} \frac12  
	\int_\Omega (u_x - v_x)^2 \,dx 
	= 
	\langle (u_t - v_t)_x, (u-v)_x \rangle 
	= 
	- \int_\Omega (u_t - v_t) (u_{xx} - v_{xx})\, dx.
\end{equation*}
Using the equations for $u_t$ and $v_t$, integrating by parts once more and using that $u_{xxx} = v_{xxx} = 0$ 
on $\partial \Omega$ we get after integration in time 
\begin{equation}   \label{eq:energy_est_difference} 
\begin{split}
	&\frac12    \int_\Omega (u_x(t) - v_x(t))^2 \,dx -  \frac12    \int_\Omega (u_x(s) - v_x(s))^2 \,dx   \\
  	= &- \int_s^t \int_\Omega  \big(  u^3|u_{xxx}|^2 + b^{\alpha-1} u^{\alpha+2} |u_{xxx}|^{\alpha+1} -   
	v^3 |v_{xxx}|^2 + b^{\alpha-1} v^{\alpha+2} |v_{xxx}|^{\alpha+1}\big) \, \big(u_{xxx} - v_{xxx}\big)  \, dx \, d\tau
\end{split}
\end{equation}
for all $0 < s < t < T$. 
Since $4 \rho > 3 + \frac1p$ the space 
$W^{4 \rho}_{p,B}(\Omega)$ embedds into $C^3(\bar \Omega)$ and we have $u,v \in C([0,T]; C^3(\bar \Omega)$. 
Thus we can easily pass to to the limit $s \downarrow 0$ and conclude  that \eqref{eq:energy_est_difference}
also holds for $s= 0$. 

Using elementary manipulations of the integrands on the right-hand side and the fact that $u(0) = v(0) = u_0$ 
we deduce the following identity for the relative energy 
\begin{equation} \label{eq:rel_energy}
\begin{aligned}
	\frac12 \int_{\Omega} \bigl(u_x(t)-v_x(t)\bigr)^2 dx
	=\ &
	-a \int_0^t \int_{\Omega} |u_{xxx} - v_{xxx}|^2 u^3  dx\, ds
	-a \int_0^t \int_{\Omega} (u_{xxx} - v_{xxx}) v_{xxx} (u^3 - v^3) dx\, ds 
	\\
	&
	-ab^{\alpha-1} \int_0^t \int_{\Omega} (u_{xxx} - v_{xxx}) 
	u^{\alpha+2} \bigl(|u_{xxx}|^{\alpha-1} u_{xxx} - |v_{xxx}|^{\alpha-1} v_{xxx}\bigr) dx\, ds \\
	& 
	-ab^{\alpha-1} \int_0^t \int_{\Omega} (u_{xxx} - v_{xxx}) |v_{xxx}|^{\alpha-1} v_{xxx} \bigl(u^{\alpha+2} - v^{\alpha+2}) dx\, ds.
\end{aligned}
\end{equation}
Since $u$ is bounded away from zero by $c>0$ we may use the inequality (cf. \cite[Lemma 4.4]{DiB93})
\begin{equation*}
	\bigl(|u_{xxx}|^{\alpha-1}u_{xxx} - |v_{xxx}|^{\alpha-1} v_{xxx}\bigr) (u_{xxx} - v_{xxx})
	\geq 
	c_{\alpha} |u_{xxx} - v_{xxx}|^{\alpha+1},
	\quad
	\alpha \geq 1,
\end{equation*}
in the third integral of the right-hand side of \eqref{eq:rel_energy} to obtain
\begin{align*}
	\frac12 \int_{\Omega} \bigl(u_x(t) &- v_x(t)\bigr)^2 dx
	+
	a c^3 \int_0^t \int_{\Omega} |u_{xxx} - v_{xxx}|^2 dx\, ds
	+
	c^{\alpha+2} c_{\alpha} \int_0^t \int_{\Omega} |u_{xxx} - v_{xxx}|^{\alpha+1} dx\, ds \\
	\leq\ &
	a \int_0^t \int_{\Omega} \bigl| (u_{xxx} - v_{xxx}) v_{xxx} (u^3 - v^3)\bigr| dx\, ds \\
	&
	+
	a b^{\alpha-1} \int_0^t \int_{\Omega} 
	\bigl| (u_{xxx} - v_{xxx}) |v_{xxx}|^{\alpha-1} v_{xxx} (u^{\alpha+2} - v^{\alpha+2}) \bigr| dx\, ds.
\end{align*} 
Applying Young's (weighted) inequality to the remaining two integrals on the right-hand side, respectively, yields
\begin{equation*}
	\int_0^t \int_{\Omega} \bigl| (u_{xxx} - v_{xxx}) v_{xxx} (u^3 - v^3) \bigr| dx\, ds 
	\leq
	\frac{c^3}{4} \int_0^t \int_{\Omega} |u_{xxx} - v_{xxx}|^2 dx\, ds 
	+
	\frac{1}{c^3} \int_0^t \int_{\Omega} |v_{xxx}|^2 |u^3 - v^3|^2 dx\, ds
\end{equation*}
and
\begin{align*}
	b^{\alpha-1} \int_0^t \int_{\Omega} & 
	\bigl| (u_{xxx} - v_{xxx}) |v_{xxx}|^{\alpha-1} v_{xxx} (u^{\alpha+2} - v^{\alpha+2}) \bigr| dx\, ds \\
	&\leq
	\frac{c^3}{4} \int_0^t \int_{\Omega} |u_{xxx} - v_{xxx}|^2 dx\, ds
	+
	\frac{b^{2(\alpha-1)}}{c^3} \int_0^t \int_{\Omega} |v_{xxx}|^{2\alpha} |u^{\alpha+2} - v^{\alpha+2}|^2 dx\, ds.
\end{align*}
Hence we find that
\begin{equation*}
	\frac{1}{2} \int_{\Omega} \bigl(u_x(t) - v_x(t)\bigr)^2 dx
	\leq
	\frac{a}{c^3} \int_0^t \int_{\Omega} |v_{xxx}|^2 |u^3 - v^3|^2 
	+
	b^{2(\alpha-1)} |v_{xxx}|^{2\alpha} |u^{\alpha+2} - v^{\alpha+2}|^2 dx\, ds.
\end{equation*}
Recalling that both $u$ and $v$ are bounded above, using the elementary inequalities
\begin{equation*}
	|u^3 - v^3|^2 \leq C(u,v) |u - v|^2
	\quad
	\text{and}
	\quad
	|u^{\alpha+2} - v^{\alpha+2}|^2 \leq C(u,v) |u - v|^2
\end{equation*}
and an $L_1$--$L_{\infty}$ H\"older estimate on $(0,t)\times\Omega$ yields
\begin{align*}
	\sup_{s \in (0,t)} \int_{\Omega} \bigl(u_x(s,x)-v_x(s,x)\bigr)^2 dx
	&\leq
	C \int_0^t \int_{\Omega} \bigl(|v_{xxx}|^2 + |v_{xxx}|^{2\alpha}\bigr) |u - v|^2 dx\, ds \\
	& \leq
	C \left(\int_0^t \int_{\Omega} \bigl(|v_{xxx}|^2 + |v_{xxx}|^{2\alpha}\bigr) dx\, ds\right) 
	||u - v||_{L_{\infty}((0,t)\times\Omega)}^2.
\end{align*}
In virtue of the Sobolev embedding $H^1(\Omega) \hookrightarrow L_{\infty}(\Omega)$ we eventually arrive at
\begin{align*}
	\sup_{s \in (0,t)} \int_{\Omega} &\bigl(u_x(s,x)-v_x(s,x)\bigr)^2 dx \\
	&\leq
	C \left(\int_0^t \int_{\Omega} \bigl(|v_{xxx}|^2 + |v_{xxx}|^{2\alpha}\bigr) dx\, ds\right) 
	\sup_{s \in (0,t)} \int_{\Omega} \bigl(u_x(s,x)-v_x(s,x)\bigr)^2 dx.
\end{align*}
Choosing $t$ small enough we obtain $u_x \equiv v_x$ and hence the assertion.
\end{proof}

\medskip

We now prove that flat films are the only possible steady state solutions of \eqref{eq:thin-film}.

\medskip

\begin{corollary}
Let $u$ be a solution of \eqref{eq:thin-film} in the sense of Theorem \ref{th:existence_Sobolev}
or Theorem \ref{th:existence_Holder} and assume that $u_t = 0$ for $t=0$. Then $u$ is constant in space,
$u = u_{\ast} \in \R_{>0}$.
\end{corollary}

\medskip

\begin{proof}
The energy inequality \eqref{eq:energy_eq} in connection with Proposition \ref{pr:dual_pairing_differentiable}
implies that
\begin{equation*}
	0 = \frac{d}{dt} E(u_{\ast}) 
	= 
	- a \int_{\Omega} u^3 |u_{xxx}|^2  + b^{\alpha-1} u^{\alpha+2} |u_{xxx}|^{\alpha+1} dx,
\end{equation*}
where the right-hand side is non-positive. Therefore the positivity of $u_{\ast}$ implies
that $(u_{\ast})_{xxx} = 0$. Thus $(u_{\ast})_{xx}$ is constant which in turn implies that $(u_{\ast})_x$ is linear. Together with the Neumann
boundary condition we obtain that $(u_{\ast})_x \equiv 0$ and hence finally that $u_{\ast} \equiv \text{const}$.
\end{proof}

\medskip


\section{Maximal time of existence}\label{sec:max_existence_time}

In this section we characterise the maximal time of existence of solutions to the non-Newtonian
thin-film equation, obtained by Theorem \ref{th:existence_Sobolev}, respectively Theorem
\ref{th:existence_Holder}. For convenience we consider only solutions in fractional 
Sobolev spaces, i.e. solutions in the sense of Theorem \ref{th:existence_Sobolev}. 
The argument may easily be adapted to the case of little H\"older functions.

\medskip

In order to state the precise result we use the same notation  
\begin{equation*}
	p \in (1,\infty), 
	\quad
	1/p < s < r < 1, 
	\quad
	\sigma = \frac{3+s}{4} 
	\quad
	\text{and}
	\quad 
	\rho = \frac{3+r}{4}
\end{equation*}
as in Theorem \ref{th:existence_Sobolev}.
Then we define for $u_0 \in W^{4\rho}_{p,B}(\Omega)$ with $u_0(x) > 0$ for all $x \in \bar{\Omega}$
\begin{equation} \label{eq:max_ex_time}
	\bar{T} = \sup\{T > 0; \text{ there exists a solution $u$ of \eqref{eq:thin-film} in the sense 
	of Theorem \ref{th:existence_Sobolev}}\}.
\end{equation}
It follows from the uniqueness result Theorem \ref{th:uniqueness} that there exists a solution
\begin{equation*}
	u \in C([0,\bar{T});W^{4\rho}_{p,B}(\Omega)) \cap C((0,\bar{T});W^4_{p,B}(\Omega))
	\cap C^1((0,\bar{T});L_p(\Omega)).
\end{equation*}

\medskip

We prove that solutions with a finite lifetime $\bar{T} < \infty$ do either converge
to zero for some point in $\bar{\Omega}$ or they blow up in every 
$W^{4\gamma}_{p,B}(\Omega)$-norm, where $\gamma \in (\sigma,1]$.

\medskip

\begin{theorem}
Suppose that $\bar{T} < \infty$. Then 
\begin{equation} \label{eq:blow-up}
	\liminf_{t \nearrow \bar{T}} 
	\frac{1}{\min_{\bar{\Omega}} u(t)} 
	+ 
	||u(t)||_{W^{4\gamma}_{p,B}(\Omega)}
	=
	\infty
\end{equation}
for all $\gamma \in (\sigma,1]$.
\end{theorem}

\medskip

\begin{proof}
This follows from the standard continuation argument. Fix $\gamma \in (\sigma,\rho]$.
Let $\bar{T} < \infty$ and assume by contradiction that \eqref{eq:blow-up} is false. 
Then there exist positive constants $r, R_{\gamma} > 0$ and a sequence $(t_n)_n$ 
with $\lim_{n\to \infty} t_n = \bar{T}$ such that
\begin{equation*}
	\min_{\bar{\Omega}}\, u(t_n) \geq r
	\quad
	\text{and}
	\quad
	||u(t_n)||_{W^{4\gamma}_{p,B}(\Omega)} \leq R_{\gamma},
	\quad
	n \in N.
\end{equation*}
By Theorem \ref{th:existence_Sobolev} 
there exists a positive time $T=T(r,R_{\gamma}) > 0$ (independent of $n$) such that for all $n \in N$ 
there exists a solution 
\begin{equation*}
	U_n \in C([0,T];W^{4\gamma}_{p,B}(\Omega)) 
	\cap 
	C^{\frac{\gamma-\sigma}{2}}([0,T];W^{4\sigma}_{p,B}(\Omega))
	\cap
	C((0,T];W^4_{p,B}(\Omega))
\end{equation*}
with initial value $U_n(0) = u(t_n)$. Here we used Theorem \ref{th:existence_Sobolev}
with $\rho = \gamma$ and $\rho' = \frac{\sigma + \rho}{2}$, whence 
\begin{equation*}
	\Acal \circ U_n \in C^{\nu}([0,T];\Hcal(W^4_{p,B}(\Omega);L_p(\Omega)))
	\quad
	\text{and}
	\quad
	\Fcal \circ U_n \in C^{\nu}([0,T];L_p(\Omega))
\end{equation*}
with $\nu = \mu \frac{\gamma - \sigma}{2}$. Moreover, $U_n(0) = u(t_n) \in W^4_{p,B}(\Omega)$
and hence by the linear theory 
$U_n \in C^1([0,T];L_p(\Omega)) \cap C([0,T];W^4_{p,B}(\Omega))$. 
Now define for $t_n \geq \bar{T} - \frac{T}{2}$
\begin{equation*}
	\tilde{u}(t)
	=
	\begin{cases}
		u(t), & 0 \leq t < t_n
		\\
		U_n(t - t_n), \quad & t_n \leq t \leq t_n + T.
	\end{cases}
\end{equation*}
As in the proof of Theorem \ref{th:max_ex_time} one can now show that $\tilde{u}$ is a solution of
$\dot{\tilde{u}} + \Acal(\tilde{u})\tilde{u} = \Fcal(\tilde{u})$ on $(0,t_n + T)$ that enjoys the regularity
\begin{equation*}
	\tilde{u} \in C([0,t_n + T];W^{4\rho}_{p,B}(\Omega)) \cap C((0,t_n + T];W^4_{p,B}(\Omega)).
\end{equation*}
Since $t_n + T \geq \bar{T} + \frac{T}{2} > \bar{T}$ this is a contradiction to the definition of 
$\bar{T}$.
\end{proof}

\medskip

Analogously one shows that for solutions in the sense of Theorem \ref{th:existence_Holder} 
the maximal time $\bar{T}$ of existence is characterised by $\bar{T} = \infty$ or
\begin{equation} \label{eq:blow-up}
	\liminf_{t \nearrow \bar{T}} 
	\frac{1}{\min_{\bar{\Omega}} u(t)} 
	+ 
	||u(t)||_{h^{4\gamma}_B(\bar{\Omega})}
	=
	\infty
\end{equation}
for all $\gamma \in (\sigma,1]$.

\medskip


\vspace{1cm}
\textsc{Acknowledgement.} 
The first author is grateful to the anonymous reviewer whose comments have improved the original version
of the manuscript and to Joachim Escher and Lorenzo Giacomelli for fruitful discussions on the topic.
This work was supported by the SFB 1060.

\medskip



\bibliographystyle{abbrv}
\bibliography{onsager}

\end{document}